\documentclass[reqno]{amsart}
\usepackage{amssymb}
\usepackage[usenames, dvipsnames]{color}
\usepackage{hyperref}
\usepackage{todonotes}
\usepackage{verbatim}
\usepackage{mathtools}

\numberwithin{equation}{section}

\newtheorem{theorem}{Theorem}[section]

\newtheorem{lemma}[theorem]{Lemma}
\newtheorem{proposition}[theorem]{Proposition}

\theoremstyle{definition}
\newtheorem{remark}[theorem]{Remark}

\theoremstyle{definition}
\newtheorem{definition}[theorem]{Definition}

\theoremstyle{definition}
\newtheorem{assumption}[theorem]{Assumption}

\makeatletter
\def\dashint{\operatorname%
{\,\,\text{\bf--}\kern-.98em\DOTSI\intop\ilimits@\!\!}}
\makeatother

\def\bH{\mathbb{H}}

\def\bN{\mathbb{N}}

\def\bR{\mathbb{R}}

\def\cD{\mathcal{D}}

\def\cF{\mathcal{F}}
\def\cG{\mathcal{G}}
\def\cH{\mathcal{H}}

\def\cM{\mathcal{M}}
\def\cN{\mathcal{N}}

\def\cS{\mathcal{S}}

\def\cU{\mathcal{U}}

\begin{document}
\title[Time fractional parabolic equations]{Time fractional parabolic equations with partially SMO coefficients}

\author[H. Dong]{Hongjie Dong}
\address[H. Dong]{Division of Applied Mathematics, Brown University, 182 George Street, Providence, RI 02912, USA}

\email{\href{mailto:Hongjie\_Dong@brown.edu}{\nolinkurl{Hongjie\_Dong@brown.edu}}}

\thanks{H. Dong was partially supported by the NSF under agreement DMS-2055244,
the Simons Foundation, grant \# 709545, and a Simons Fellowship.}

\author[D. Kim]{Doyoon Kim}
\address[D. Kim]{Department of Mathematics, Korea University, 145 Anam-ro, Seongbuk-gu, Seoul, 02841, Republic of Korea}

\email{\href{mailto:doyoon_kim@korea.ac.kr}{\nolinkurl{doyoon_kim@korea.ac.kr}}}

\thanks{D. Kim were supported by the National Research Foundation of Korea (NRF) grant funded by the Korea government (MSIT) (2019R1A2C1084683).}

\subjclass[2010]{35R11, 26A33, 35R05}

\keywords{parabolic equations, time fractional derivative, measurable coefficients, small mean oscillations, Muckenhoupt weights}

\begin{abstract}
We present the unique solvability in Sobolev spaces of time fractional parabolic equations in divergence and non-divergence forms.
The leading coefficients are merely measurable in $(t,x_1)$ for $a^{ij}$, $1 \leq i,j \leq d$, $(i,j) \neq (1,1)$.
The coefficient $a^{11}$ is merely measurable locally either in $t$ or $x_1$.
As functions of the remaining variables, the coefficients have small mean oscillations.
We consider mixed norm Sobolev spaces with Muckenhoupt weights. Our results generalize previous work on parabolic equations with time fractional derivatives to a much larger class of coefficients and solution spaces.
\end{abstract}

\maketitle

\section{Introduction}

We consider parabolic equations in divergence and non-divergence forms with time fractional derivatives.
The equations are of the form
\[
-\partial_t^\alpha u + D_i(a^{ij}D_j u + a^i u) + b^i D_i u + c u = D_i g_i + f
\]
and
\[
-\partial_t^\alpha u + a^{ij} D_{ij} u + b^i D_i u + c u =f
\]
in $(0,T) \times \bR^d =: \bR^d_T$, where  $\partial_t^\alpha u$ is the Caputo fractional time derivative of order $\alpha \in (0,1)$ defined by
\[
\partial_t^\alpha u(t,x) = \frac{1}{\Gamma(1-\alpha)} \frac{d}{dt} \int_0^t (t-s)^{-\alpha} \left[ u(s,x) - u(0,x) \right] \, ds
\]
for sufficiently smooth $u(t,x)$.

In our previous paper \cite{MR4345837}, it was proved that there exist unique solutions to the above equations when $g_i, f \in L_p(\bR^d_T)$ with $p \in (1,\infty)$ and $b^i = c = 0$.
The solutions in \cite{MR4345837} are such that $u, Du \in L_p(\bR^d_T)$ and $\partial_t^\alpha u \in \bH_p^{-1}(\bR^d_T)$ (see Section \ref{sec02} for the definition of $\bH_p^{-1}$) for equations in divergence form and $u, Du, D^2u, \partial_t^\alpha u \in L_p(\bR^d_T)$ for equations in non-divergence form with zero initial conditions.
The novelty of the paper \cite{MR4345837} is that the coefficient matrix $\{a^{ij}\}_{i,j=1,\ldots,d}$ is allowed to have no regularity assumptions as a function of the temporal and one spatial variable except one of the diagonal coefficients.
Since the class of coefficients in this paper is the very generalization of \cite{MR4345837}, let us give a more detailed description of the coefficients there.
The coefficients $a^{ij}$ are functions of $(t,x_1) \in \bR \times \bR$ without any regularity assumptions (i.e., merely measurable), only satisfying the ellipticity and boundedness condition (see \eqref{eq0310_01}) for all $i,j=1,\ldots,d$.
The coefficient $a^{11}$ has a restriction that it needs to be either $a^{11} = a^{11}(t)$ or $a^{11} = a^{11}(x_1)$ with no regularity assumptions.
Note that, in the parabolic case with the usual time derivative $u_t$, if $a^{ij}(t,x_1)$ have no regularity assumptions for all $i,j=1,\ldots,d$, there is no unique solvability of parabolic equations in Sobolev spaces for $p \in (1,3/2)$ or $p \in (3,\infty)$ even when $d=1$.
See \cite{MR3488249} for a counterexample.
Thus, the coefficients in \cite{MR4345837} are optimal in the sense that the aforementioned restriction on $a^{11}$ cannot be removed.

In this paper, we generalize previous results on parabolic equations with time fractional derivatives to a much larger class of coefficients and solution spaces.
As to solution spaces, we consider weighted Sobolev spaces with mixed norms.
See \eqref{eq0211_02}.
The weights are of the form $w(t,x) = w_1(t)w_2(x)$, where $w_1(t) \in A_p(\bR)$ and $w_2 (x) \in A_q(\bR^d)$.
Here, $A_p(\cdot)$ means a collection of Muckenhoupt weights.
See \eqref{eq0211_03}.
Such weighted Sobolev spaces are also considered in \cite{MR4186022} for non-divergence form equations, but the coefficients $a^{ij}(t,x)$ there are limited to those being measurable only in the temporal variable and having small mean oscillations with respect to all the spatial variables $x \in \bR^d$.
In contrast, the coefficients $a^{ij} = a^{ij}(t,x)$ in this paper are functions of $(t,x_1) \in \bR \times \bR$ (a function of $t$ or $x_1$ for $a^{11}$) with no regularity assumptions, and have small mean oscillations as functions of the remaining variables $x' \in \bR^{d-1}$ ($x \in \bR^d$ or $(t,x') \in \bR \times \bR^{d-1}$ for $a^{11}$), where $x' = (x_2,\ldots,x_d) \in \bR^{d-1}$ for $d = 2, 3, \ldots$.
See Assumption \ref{assum0808_1}.
Note that if the coefficients $a^{ij}$, $(i,j) \neq (1,1)$, are functions of only $(t,x_1)$, they are in the same class of coefficients as in \cite{MR4345837}, but the coefficient $a^{11}$ is more general than in \cite{MR4345837} even if it is a function of only $t$ or $x_1$.
We will further discuss the assumption on the coefficient $a^{11}$ later.
One advantage of considering such coefficients is that by using even/odd extensions, we immediately obtain the corresponding results for equations in the half space $\{x_1>0\}$ with either the zero Dirichlet, conormal derivative, or Neumann boundary condition, which generalize the main result of \cite{MR4186022}.
By a partition of unity argument, the results can be further extended to cylindrical domains with appropriate conditions on $a^{ij}$.
See Remark \ref{rem2.10}.

To establish the results for equations in {\em weighted parabolic Sobolev spaces}, we utilize so-called $L_{p_0}$-mean oscillation estimates, $p_0 \in (1,\infty)$, for solutions to equations with coefficients considered in \cite{MR4345837}.
See Propositions \ref{prop0930_01} and \ref{prop0930_02}, where no weights or mixed norms are involved.
Once $L_{p_0}$-mean oscillation estimates are obtained, particularly for $p_0$ sufficiently close to $1$, one can prove the unique solvability along with appropriate estimates in {\em weighted Sobolev spaces with mixed norms}.
See the proofs of Theorems \ref{thm1122_1} and \ref{thm1122_2} below as well as \cite{MR3812104} and the references therein.
The $L_{p_0}$-mean oscillation estimates are derived from the unique solvability and a priori $L_p$-estimates for equations in Sobolev spaces with {\em unmixed norms and no weights}.
In this respect, the paper \cite{MR4345837} can be considered a prequel to this paper.
That is, we establish a chain of results
\[
{\small\text{unmixed $L_p$ estimates} \Rightarrow \text{$L_{p_0}$-mean oscillation estimates} \Rightarrow \text{weighted $L_{p,q}$-estimates},}
\]
where \cite{MR4345837} takes care of the first result and this paper resolves the last two results.
This chain of implications also applies to the usual parabolic equations with the local time derivative $u_t$.
However, in the time fractional parabolic case, due to the presence of the non-local time derivative, the first implication of the above chain for $p_0$ close to $1$ is not possible if one follows the proofs for the usual parabolic case presented, for instance, in \cite{MR3812104}.
Specifically, one of the main steps to prove the first implication is improving the regularity of solutions to homogeneous equations.
In the time fractional parabolic case, this regularity-improving process is limited due to the non-Markovian nature of the time fractional derivative.
That is, the regularity of a function at the current moment is affected by the whole history of the function.
To overcome this difficulty, we adapt the approach from \cite{MR4186022}, one of the main features of which is decomposing solutions of equations in infinite cylinders of the form $(-\infty, t) \times B_R(x)$ instead of the usual parabolic cylinders $(t-R_1,t) \times B_{R_2}(x)$.
Moreover, we develop the approach with a refined decomposition of solutions to obtain mean oscillation estimates of $D_1 u$ and $D_{1j} u$, $j=2,\ldots,d$, (for equations in divergence and non-divergence forms, respectively), which have not appeared in the previous papers.

Let us provide further details on the mean oscillation estimates of $D_1 u$ and $D_{1j} u$, $j=2,\ldots,d$, and make a remark on the coefficient $a^{11}$.
We first recall that there are some previous results for parabolic equations with the usual local time derivative $u_t$ and coefficients $a^{ij}$ similar to those in this paper.
Our results in this paper can be compared to those in \cite{MR2764911} for the usual parabolic equations in divergence form and in \cite{MR2833589} for the usual parabolic equations in non-divergence form.
However, these previous results are confined to {\em unmixed} $L_p$ spaces with no weights.
In \cite{MR2764911, MR2833589}, the desired results are derived from, for instance, in the divergence type equation case, the mean oscillation estimates of only $D_j u$, $j=2,\ldots,d$, with $D_1 u$ excluded due to the lack of regularity assumptions on the coefficients in $(t,x_1)$.
These mean oscillation estimates imply the $L_p$-estimates of $D_j u$, $j=2,\ldots,d$, which in turn proves the $L_p$-estimate of $D_1u$ by a sophisticated scaling argument (see, for instance, \cite[Lemma 3.4]{MR2764911} or \cite[Lemma 3.4]{MR4345837}).
For equations in weighted Sobolev spaces, the scaling argument is unavailable because it essentially relies on the scaling invariance (up to a multiple of a constant) of unweighted $L_p$-norms.
Thus, we must deal with the mean oscillation estimates of $D_1u$ in the divergence case and $D_{1j}u$, $j=2,\ldots,d$, in the non-divergence case.
To obtain such estimates,  we use a new approach based on, as mentioned above, a refined composition $u=w+v = w + \tilde{v} + \hat{v}$, where $v$ is a solution to a homogeneous equation and $\hat{v}$ is a solution to a simpler homogeneous equation. See Lemma \ref{lem0228_1}.
The mean oscillation estimates for $D_1u$ and $D_{1j}u$, $j=2,\ldots,d$, allow us to have a more general assumption on $a^{11}$ than those in \cite{MR2764911, MR2833589}.
The difference is that the coefficient $a^{11}$ in \cite{MR2764911, MR2833589} can be measurable either in $t$ or $x_1$ globally in the whole domain, while in this paper $a^{11}$ can be measurable in $t$ or $x_1$ locally.
That is, $a^{11}$ can be measurable in $t$ in one region of the domain and measurable in $x_1$ in another region of the domain.
See Assumption \ref{assum0808_1}.
Applying the arguments and the assumption on $a^{11}$ in this paper to equations with $u_t$, we can get similar results for the usual parabolic equations, which are also new.
It is also worth noting that we derive the necessary results for equations in non-divergence form from those for equations in divergence form so that we do not need to deal with equations in two different forms separately.
The techniques developed in this paper might also be applicable to other types of equations with non-local operators.

To provide context for our work on time fractional parabolic equations and related results in the literature, we refer the reader to the paper \cite{MR4345837} and the references therein.
Also see \cite{MR4247129}, where the authors deal with equations similar to those in this paper but in a different type of weighted Sobolev spaces with $\alpha \in (0,2)$ and continuous $a^{ij}(t,x)$.
Further, one can find related results on time fractional evolution equations in Hilbert space settings in \cite{MR2538276,MR3814402, MR4200127, MR4040846}.

The remainder of the paper is organized as follows. In the next section, we
introduce necessary notation and state the main results of the paper. In Section \ref{sec03}, we derive estimates for equations in divergence form when the coefficients $a^{ij}$ are functions of $(t,x_1)$ ($a^{11}$ is a function of either $t$ or $x_1$).
We then use these results in Section \ref{sec04} to prove mean oscillation estimates of $Du$ for the divergence case and $D^2u$ (except $D_1^2u$) for the non-divergence case.
In Section \ref{sec05}, we prove our main theorems.
Finally, in the Appendix, we present an inequality necessary to take care of the non-local property of time fractional derivatives.

\section{Notation and Main results}							\label{sec02}

\subsection{Notation}
We define the parabolic cylinders
$$
Q_{R_1,R_2}(t,x) = (t-R_1^{2/\alpha}, t) \times B_{R_2}(x), \quad Q_R(t,x) = Q_{R,R}(t,x).
$$
For $\Omega \subset \bR^d$ and $T > 0$, we denote
$\Omega_T = (0,T) \times \Omega$. In particular, we have $\bR^d_T = (0,T) \times \bR^d$.
We write
$$
D_{x'} u = D_{x_\ell} u, \quad \ell = 2, \ldots, d.
$$
We use the notation $(u)_{\cD}$ to denote the average of $u$ over $\cD$, where $\cD$ is a subset of $\bR^{d+1}$.
That is,
\[
(u)_\cD = \dashint_{\cD} u(t,x) \, dx \, dt = \frac{1}{|\cD|} \int_{\cD} u(t,x) \, dx \, dt.
\]

Let $\alpha \in (0,1)$ and $S \in \bR$.
We denote the $\alpha$-the integral of $u$ with origin $S$ by
\[
I_S^{\alpha} u = \frac{1}{\Gamma(1-\alpha)} \int_S^t (t-s)^{\alpha-1} u(s,x) \, ds.
\]
In particular, we write $I^{\alpha} u$ if $S = 0$.
Set
\[
\partial_t^\alpha u = \partial_t I_S^{1-\alpha} \left( u(\cdot,x) - u(S,x) \right) = \frac{1}{\Gamma(1-\alpha)} \partial_t \int_S^t (t-s)^{-\alpha}\left(u(s,x) - u(S,x)\right) \, ds
\]
for a sufficiently smooth $u$, which is called the Caputo fractional derivative of order $\alpha$ with respect to time at $S$ (time fractional derivative of order $\alpha$ at $S$).
If $u$ further satisfies $u(S,x) = 0$, we see that
$$
\partial_t^\alpha u
= \partial_t I_S^{1-\alpha} u.
$$
Whenever we write $\partial_t^\alpha u$, the origin $S$ is clear from the context or $S=0$, that is,
\[
\partial_t^\alpha u = \partial_t I^{1-\alpha}_0 u = \partial_t I^{1-\alpha} u.
\]

\subsection{Function spaces}

Here we introduce function spaces for solutions to the equations discussed in this paper.
We set $p \in (1,\infty)$ and $\alpha \in (0,1)$.

For $p \in (1,\infty)$ and $k \in \{1,2,\ldots\}$, we let $A_p(\bR^k,dx)$ be the set of all locally integrable non-negative functions $w$ on $\bR^k$ such that
\begin{equation}
							\label{eq0211_03}
[w]_{A_p} := \sup_{x_0 \in \bR^k, r > 0} \left(\dashint_{B_r(x_0)} w(x) \, dx\right) \left( \dashint_{B_r(x_0)} w(x)^{-\frac{1}{p-1}} \, dx \right)^{p-1} < \infty,
\end{equation}
where $B_r(x_0) = \{x \in \bR^k : |x-x_0| < r\}$. Recall that $[w]_{A_p} \geq 1$.

For $w(t,x) = w_1(t)w_2(x)$, where $(t,x) \in \bR \times \bR^d$, $w_1 \in A_p(\bR,dt)$, and $w_2 \in A_q(\bR^d,dx)$, we set $L_{p,q,w}(\bR^d_T)$ to be the set of all measurable functions $f$ defined on $\bR^d_T$ such that
\begin{equation}
							\label{eq0211_02}
\|f\|_{L_{p,q,w}(\bR^d_T)} :=
\left(\int_0^T \left( \int_{\bR^d} |f(t,x)|^q w_2(x) \, dx \right)^{p/q} w_1(t) \, dt \right)^{1/p}  < \infty.
\end{equation}
Note that if $f \in L_{p,q,w}(\bR^d_T)$, by the reverse H\"older's inequality for $A_p$ weights (see, for instance, Corollary 7.2.6 and Remark 7.2.3 in \cite{MR3243734}), there exists $p_0 \in (1,\infty)$ such that $f \in L_{p_0}\left((0,T) \times B_R\right)$ for any $R > 0$.

\begin{definition}
							\label{def0513_1}
We say $u \in \bH_{p,q,w,0}^{\alpha,0}(\Omega_T)$ if $u \in L_{p,q,w}(\Omega_T)$ and there exists $f \in L_{p,q,w}(\Omega_T)$ such that
\begin{equation}
							\label{eq0513_01}
\int_{\Omega_T}I^{1-\alpha} u \, \varphi_t \, dx \, dt = - \int_{\Omega_T} f \, \varphi \, dx \, dt
\end{equation}
for all $\varphi \in C_0^\infty\left([0,T) \times \Omega\right)$.
In this case, as a weak derivative, we have $\partial_t I^{1-\alpha} u = f$. We also have
\begin{equation}
							\label{eq1115_02}
\partial_t^\alpha u = \partial_t I^{1-\alpha} u.
\end{equation}
See Remark \ref{rem0727_1}.
For solutions to non-divergence type equations, for a positive integer $k$, mostly $k=2$, we set
\[
\bH_{p,q,w,0}^{\alpha,k}(\Omega_T) = \{u \in \bH_{p,q,w,0}^{\alpha,0}(\Omega_T): D^j u \in L_{p,q,w}(\Omega_T), j=1,\ldots,k\}
\]
with the norm
\[
\|u\|_{\bH_{p,q,w,0}^{\alpha,k}(\Omega_T)} = \sum_{j=0}^k\|D^j u\|_{L_{p,q,w}(\Omega_T)} + \|\partial_t^\alpha u\|_{L_{p,q,w,0}(\Omega_T)}.
\]

\end{definition}

\begin{remark}
							\label{rem0727_1}
Note that the test function $\varphi$ in \eqref{eq0513_01} belongs to $C_0^\infty([0,T) \times \Omega)$ so that $\varphi(0,x)$ is not necessarily zero.
The equality \eqref{eq0513_01} for such test functions implies that the equality \eqref{eq1115_02} holds for all $u \in \bH_{p,q,w,0}^{\alpha,0}(\Omega_T)$.
Precisely, as shown in \cite{preprint2022}, for any $u \in \bH_{p,q,w,0}^{\alpha,0}(\Omega_T)$ ($\bH_{p,q,w,0}^{\alpha,k}(\Omega_T)$ as well), there exists an approximating sequence $\{u_n\}$ of $u$ such that $u_n \in C^\infty([0,T] \times \Omega)$, $u_n$ vanishes for large $|x|$ (when $\Omega$ is unbounded),  and $u_n(0,x) = 0$. Thus, the equality \eqref{eq1115_02} makes sense as
\[
\partial_t^\alpha u = \lim_{n \to \infty} \partial_t^\alpha u_n = \lim_{n \to \infty}  \partial_t I^{1-\alpha}u_n(t,x),
\]
where the limit is in the norm of $L_{p,q,w}(\Omega_T)$.
If $\alpha \in (1/p,1)$, for which the initial trace $u(0,x)$ makes sense, the equality \eqref{eq0513_01} implies that $u(0,x)$ is zero.
Similarly, by \eqref{eq0601_01} below,
the equality  \eqref{eq1115_03} makes sense.
 For initial traces and other related results, see \cite{preprint2022}.
\end{remark}

To introduce function spaces for solutions to divergence type equations, we first define $\bH_{p,q,w}^{-1}(\Omega_T)$ as follows.
A distribution $v$ on $\Omega_T$ is said to $v \in \bH_{p,q,w}^{-1}(\Omega_T)$ if there exist $G_i, F \in L_{p,q,w}(\Omega_T)$ such that
\begin{equation*}
v = D_i G_i + F
\end{equation*}
in the distribution sense.
We define the norm of $\bH_{p,q,w}^{-1}(\Omega_T)$ by
\[
\|v\|_{\bH_{p,q,w}^{-1}(\Omega_T)} = \inf \{ \|F\|_{L_{p,q,w}(\Omega_T)} + \|G_i\|_{L_{p,q,w}(\Omega_T)}: v = D_i G_i + F\}.
\]

\begin{definition}
							\label{def0727_1}
We say $u \in \cH_{p,q,w,0}^{\alpha,0}(\Omega_T)$ if $u \in L_{p,q,w}(\Omega_T)$ and there exist $g_i, f \in L_{p,q,w}(\Omega_T)$ such that
\begin{equation}
							\label{eq0601_01}
\int_{\Omega_T} I^{1-\alpha} u \, \varphi_t \, dx \, dt
= \int_{\Omega_T} \left(g_i D_i \varphi - f \varphi \right) \, dx \, dt
\end{equation}
for all $\varphi \in C_0^\infty\left([0,T) \times \Omega\right)$.
That is, in the distributional sense,
\[
\partial_t I^{1-\alpha} u = D_i g_i + f
\]
and $\partial_t I^{1-\alpha} u \in \bH_{p,q,w}^{-1}(\Omega_T)$.
As mentioned in Remark \ref{rem0727_1}, we have
\begin{equation}
							\label{eq1115_03}
\partial_t^\alpha u = \partial_t I^{1-\alpha} u.
\end{equation}

For solutions to divergence type equations, we set
\[
\cH_{p,q,w,0}^{\alpha,1}(\Omega_T) = \{ u \in \cH_{p,q,w,0}^{\alpha,0}(\Omega): Du \in L_{p,q,w}(\Omega_T)\}
\]
with the norm
\begin{equation*}
\|u\|_{\cH_{p,q,w,0}^{\alpha,1}(\Omega_T)} = \|u\|_{L_{p,q,w}(\Omega_T)} + \|Du\|_{L_{p,q,w}(\Omega_T)} + \|\partial_t^\alpha u\|_{\bH_{p,q,w}^{-1}(\Omega_T)}.
\end{equation*}

\end{definition}

As usual, when $p=q$ and $w(t,x) = 1$, we denote
\[
\bH_{p,p,1,0}^{\alpha,2}(\Omega_T) = \bH_{p,0}^{\alpha,2}(\Omega_T) \quad \text{and} \quad \cH_{p,p,1,0}^{\alpha,1}(\Omega_T) = \cH_{p,0}^{\alpha,1}(\Omega_T).
\]

\begin{remark}
In our previous papers, in particular, when $p=q$ and $w=1$, we used spaces such as $\widetilde{\bH}_p^{\alpha,2}(\Omega_T)$, $\bH_p^{\alpha,2}(\Omega_T)$, $\bH_{p,0}^{\alpha,2}(\Omega_T)$ for non-divergence type equations (see, for instance, \cite{MR3899965}), and $\widetilde{\cH}_p^{\alpha,1}(\Omega_T)$, $\cH_p^{\alpha,1}(\Omega_T)$, $\cH_{p,0}^{\alpha,1}(\Omega_T)$ for divergence type equations (see, for instance, \cite{MR4030286}).
In this paper even if we use the same notation, some of the spaces from \cite{MR3899965} and \cite{MR4030286} differ from those in this paper.
Indeed, as mentioned in Remark \ref{rem0727_1}, $u \in \bH_{p,0}^{\alpha,k}(\Omega_T)$ can be approximated by infinitely differentiable functions with zero initial values.
This means $\bH_{p,0}^{\alpha,k}(\Omega_T)$ exactly corresponds to the space using the same notation, for instance, in \cite{MR3899965}.
The same applies to $\cH_{p,0}^{\alpha,1}(\Omega_T)$ defined, for instance, in \cite{MR4030286}.
However, the space $\bH_p^{\alpha,k}(\Omega_T)$ in \cite{MR3899965} turns out to be the same as $\bH_{p,0}^{\alpha,k}(\Omega_T)$.
Similarly,  $\cH_p^{\alpha,1}(\Omega_T)$ in \cite{MR4030286} is the same as $\cH_{p,0}^{\alpha,1}(\Omega_T)$.
We do not use $\widetilde{\bH}_p^{\alpha,k}(\Omega_T)$ and $\widetilde{\cH}_p^{\alpha,1}(\Omega_T)$ in this paper.
It is worth pointing out that, by the definitions in \cite{MR3899965}, to verify that $u \in \bH_{p,0}^{\alpha,k}(\Omega_T)$, one must find an approximating sequence $\{u_n\}$ such that $u_n \in C^\infty([0,T] \times \Omega)$ with $u_n(0,x)=0$.
However, by Definition \ref{def0513_1}, we now only need to check if the equality \eqref{eq0513_01} holds for all test functions from $C_0^\infty([0,T) \times \Omega)$.
Similarly, to check that $u \in \cH_{p,0}^{\alpha,1}(\Omega_T)$, we only need to verify \eqref{eq0601_01} for all $\varphi \in C_0^\infty([0,T) \times \Omega)$.
\end{remark}

\subsection{Assumptions}

Throughout the paper, we assume that there exists $\delta \in (0,1)$ such that
\begin{equation}
							\label{eq0310_01}
a^{ij}(t,x)\xi_i \xi_j \geq \delta |\xi|^2, \quad |a^{ij}(t,x)| \leq \delta^{-1}
\end{equation}
for any $\xi \in \bR^d$ and $(t,x) \in \bR \times \bR^d$.

To state our regularity assumptions on $a^{ij}$, we first introduce coefficients $a^{ij}$ which are measurable in $(t,x_1)$ except $(i,j) = (1,1)$.
For $a^{11}$, we have either $a^{11} = a^{11}(t)$ or $a^{11}=a^{11}(x_1)$.

\begin{assumption}
							\label{assum0808_2}
The coefficient matrix $\{a^{ij}\}_{i,j=1,\ldots,d}$ with the ellipticity (and boundedness) condition \eqref{eq0310_01} satisfies either (i) or (ii) of the following.
\begin{enumerate}
\item[(i)]
$a^{11} = a^{11}(t)$, $a^{ij} = a^{ij}(t,x_1)$ for $(i,j) \neq (1,1)$.
\item[(ii)]
$a^{11} = a^{11}(x_1)$, $a^{ij} = a^{ij}(t,x_1)$ for $(i,j) \neq (1,1)$.
\end{enumerate}
\end{assumption}

Here are our assumptions for  partially SMO coefficients.
As mentioned above, $a^{ij}$ always satisfy \eqref{eq0310_01}.
We also impose the boundedness assumption for lower-order coefficients.

For $x \in \bR^d$, we write $x = (x_1,x')$, where $x_1\in \bR$ and $x'\in \bR^{d-1}$.
We then denote
\[
B_r'(x') = \{y' \in \bR^{d-1}: |x'-y'| < r\}, \quad
Q_r'(t,x') = (t-r^{2/\alpha},t) \times B_r'(x').
\]

\begin{assumption}[$\gamma_0$]
							\label{assum0808_1}
There is a constant $R_0 \in (0,1]$ satisfying the following.
\begin{enumerate}
\item For each $(t_0,x_0) \in \bR^{d+1}$ and $r \in (0,R_0]$, the coefficients $a^{ij}$ with $(i,j) \neq (1,1)$ satisfy
\[
\dashint_{Q_r(t_0,x_0)} |a^{ij}(t,x_1,x') - \bar{a}^{ij}(t,x_1) | \, dx \, dt \leq \gamma_0,
\]
where
\begin{equation}
							\label{eq1122_01}
\bar{a}^{ij}(t,x_1) = \dashint_{B_r'(x_0')} a^{ij}(t,x_1,y') \, dy'.
\end{equation}
\item For each $(t_0,x_0) \in \bR^{d+1}$, the coefficient $a^{11}$ satisfies either (2.i) or (2.ii) of the following.
\begin{enumerate}
\item[(2.i)] For $(t_1,x_0)$ with $t_1 \in (t_0-3 R_0^{2/\alpha},t_0]$ and $r \in (0,R_0]$,
\begin{equation*}
\dashint_{Q_r(t_1,x_0)} |a^{11}(t,x_1,x') - \bar{a}^{11}(t)| \, dx \, dt \leq \gamma_0,
\end{equation*}
where
\begin{equation}
							\label{eq1122_02}
\bar{a}^{11}(t) = \dashint_{B_r(x_0)} a^{11}(t,y) \, dy.
\end{equation}
\item[(2.ii)] For $(t_1,x_0)$ with $t_1 \in (t_0-3 R_0^{2/\alpha},t_0]$ and $r \in (0,2^{\alpha/2}R_0]$,
\begin{equation*}
\dashint_{Q_r(t_1,x_0)} |a^{11}(t,x_1,x') - \bar{a}^{11}(x_1)| \, dx \, dt \leq \gamma_0,
\end{equation*}
where
\begin{equation}
							\label{eq1122_03}
\bar{a}^{11}(x_1) = \dashint_{Q_r'(t_1,x_0')} a^{11}(s,x_1,y') \, dy' \, ds.
\end{equation}
\end{enumerate}
\end{enumerate}
For the lower-order coefficients $a^i$, $b^i$, and $c$, there exists $K_0 \geq 0$ such that
\[
|a^i| \leq K_0, \quad |b^i| \leq K_0, \quad |c| \leq K_0.
\]
\end{assumption}

\begin{remark}
							\label{rem1122_1}
Under Assumption \ref{assum0808_1}, we observe that, for any $R \in (r^{2/\alpha}, 2R_0^{2/\alpha}]$ and $r \leq R_0$,
\[
\dashint_{\!t_0-R}^{\,\,\,t_0} \dashint_{B_r(x_0)} |a^{ij}-\bar{a}^{ij}| \, dx \, dt \leq 2 \gamma_0
\]
for $a^{ij}$ with $(i,j) \neq (1,1)$ and for $a^{11}$ satisfying (2.i), where $\bar{a}^{ij}$ is in \eqref{eq1122_01} for $a^{ij}$ with $(i,j)\neq(1,1)$ and in \eqref{eq1122_02} for $a^{11}$.
For $a^{11}$ satisfying (2.ii), we have
\[
\dashint_{\!t_0-R}^{\,\,\,t_0} \dashint_{B_r(x_0)} |a^{11}-\bar{a}^{11}| \, dx \, dt \leq N \frac{R}{r^{2/\alpha}} \gamma_0,
\]
where $\bar{a}^{11}$ is in \eqref{eq1122_03} and $N = N(d,\alpha)$.
For the proofs of the above inequalities, see  \cite[Remark 2.3]{MR3899965} and \cite[Lemma 2.14]{MR4464544}.
\end{remark}

\subsection{Main results}

Our first main theorem is for equations in divergence form.
By a solution $u$ to \eqref{eq1124_02}, we mean that $u$ satisfies
\begin{multline*}
\int_0^T\int_{\bR^d}\left( I^{1-\alpha} u \, \varphi_t \, dx \, dt - a^{ij}D_j u D_i \varphi - a^i u D_i \varphi + b^i D_i u \varphi + c u \varphi \right) \, dx \, dt
\\
= \int_0^T\int_{\bR^d} (f \varphi - g_i D_i \varphi) \, dx \, dt
\end{multline*}
for any $\varphi \in C_0^\infty\left([0,T) \times \bR^d\right)$.
As discussed in Remark \ref{rem0727_1}, by the definition of $\cH_{p,q,w,0}^{\alpha,1}(\bR^d_T)$ (the above formulation as well), \eqref{eq1124_02} is an equation with the zero initial condition if the initial trace makes sense.

\begin{theorem}[Divergence case]
							\label{thm1122_1}
Let $\alpha \in (0,1)$, $T \in (0,\infty)$, $p,q \in (1,\infty)$, $K_1 \in [1,\infty)$, and $w(t,x) = w_1(t)w_2(x)$, where
\[
w_1(t) \in A_p(\bR,dt), \quad w_2(x) \in A_q(\bR^d,dx), \quad [w_1]_{A_p} \leq K_1, \quad [w_2]_{A_q} \leq K_1.
\]
Then, there exists $\gamma_0 = \gamma_0(d,\delta,\alpha, p, q, K_1) \in (0,1)$ such that, under Assumption \ref{assum0808_1} ($\gamma_0$), the following hold.

For any $u \in \cH_{p,q,w,0}^{\alpha,1}(\bR^d_T)$ satisfying
\begin{equation}
							\label{eq1124_02}
-\partial_t^\alpha u + D_i(a^{ij}D_j u + a^i u) + b^i D_i u + c u = D_i g_i + f
\end{equation}
in $\bR^d_T$, where $g_i, f \in L_{p,q,w}(\Omega_T)$, we have
\begin{equation}
							\label{eq1124_03}
\|u\|_{\cH_{p,q,w,0}^{\alpha,1}(\bR^d_T)} \leq N \|g_i\|_{L_{p,q,w}(\bR^d_T)} + N \|f\|_{L_{p,q,w}(\bR^d_T)},
\end{equation}
where $N = N(d,\delta,\alpha, p,q,K_1,K_0,R_0,T)$.
Moreover, for $g_i, f \in L_{p,q,w}(\bR^d_T)$, there exists a unique solution $u \in \cH_{p,q,w,0}^{\alpha,1}(\bR^d_T)$ satisfying \eqref{eq1124_02}.
\end{theorem}

Here is our main theorem for equations in non-divergence form.
The equation \eqref{eq1124_01} holds almost everywhere and, by the definition of $\bH_{p,q,w,0}^{\alpha,2}(\bR^d_T)$, has the zero initial condition if the initial trace makes sense.

\begin{theorem}[Non-divergence case]
							\label{thm1122_2}
Let $\alpha$, $T$, $p$, $q$, $K_1$, and $w$ be as in Theorem \ref{thm1122_1}.
Then, there exists $\gamma_0 = \gamma_0(d,\delta,\alpha, p, q, K_1) \in (0,1)$ such that, under Assumption \ref{assum0808_1} ($\gamma_0$), the following hold.

For any $u \in \bH_{p,q,w,0}^{\alpha,2}(\bR^d_T)$ satisfying
\begin{equation}
							\label{eq1124_01}
-\partial_t^\alpha u + a^{ij} D_{ij} u + b^i D_i u + c u =f
\end{equation}
in $\bR^d_T$, where $f \in L_{p,q,w}(\Omega_T)$, we have
\begin{equation}
							\label{eq1124_04}
\|u\|_{\bH_{p,q,w,0}^{\alpha,2}(\bR^d_T)} \leq N \|f\|_{L_{p,q,w}(\bR^d_T)},
\end{equation}
where $N = N(d,\delta,\alpha, p,q,K_1,K_0,R_0,T)$.
Moreover, for $f \in L_{p,q,w}(\bR^d_T)$, there exists a unique $u \in \bH_{p,q,w,0}^{\alpha,2}(\bR^d_T)$ satisfying \eqref{eq1124_01}.
\end{theorem}

\begin{remark}
                \label{rem2.10}
By using even/odd extensions, from Theorems \ref{thm1122_1} and \ref{thm1122_2}, we can readily obtain the corresponding results in the half space $\{x_1>0\}$ with either the zero Dirichlet (for equations in divergence and non-divergence forms), conormal derivative (for equations in divergence form), or Neumann boundary condition (for equations in non-divergence form). We refer the reader to the proofs of \cite[Theorems 2.4 and 2.5]{MR2764911} for details.
We remark that, as to equations on sufficiently regular domains other than the whole Euclidean space and a half space, one can deal with parabolic equations with $a^{ij}$ measurable in $t$ or in one spatial variable (not in both $t$ and one spatial variable as those in this paper).
In particular, near the boundary, the spatial direction in which $a^{ij}$ are measurable has to be (almost) perpendicular to the boundary.
Also, see \cite{MR3206986} for parabolic equations (with $u_t$) in non-divergence form with a restricted range of $p$ when $a^{ij}$ are measurable in a tangential direction to the boundary.
\end{remark}

\section{Equations in divergence form with measurable coefficients}
							\label{sec03}

In this section we consider
\[
\partial_t^\alpha u + D_i(a^{ij}D_ju) = D_i g_i + f
\]
with coefficients $a^{ij}$ satisfying Assumption \ref{assum0808_2}.

\begin{proposition}[Right-hand side having less summability]
							\label{prop1113_1}
Let $\alpha \in (0,1)$, $p \in (1,\infty)$, $T \in (0,\infty)$, and $a^{ij}$ satisfy Assumption \ref{assum0808_2}.
Then, for $g_i \in L_p(\bR^d_T)$ and $f \in L_q(\bR^d_T)$, where $q \in (1,\infty)$ and
$$
\frac{1}{d+2/\alpha}+ \frac{1}{p} \geq \frac{1}{q} \geq \frac{1}{p},
$$
there exists a unique function $u$ on $\bR^d_T$ such that $u, Du \in L_p(\bR^d_T)$ and
\begin{equation}
							\label{eq1113_03}
-\partial_t^\alpha u + D_i\left(a^{ij} D_j u \right) = D_i g_i + f
\end{equation}
in $\bR^d_T$ with the estimate
\begin{equation}
							\label{eq1113_04}
\|u\|_{L_p(\bR^d_T)} + \|Du\|_{L_p(\bR^d_T)} \leq N \|g_i\|_{L_p(\bR^d_T)} + N\|f\|_{L_q(\bR^d_T)},
\end{equation}
where $N_0 = N_0(d,\delta,\alpha,p,q,T)$.
\end{proposition}

\begin{proof}
Find a sequence $\{f^k\}$ such that
$f^k \in L_p \cap L_q(\bR^d_T)$ and $f_k \to f$ in $L_q(\bR^d_T)$.
By using \cite[Theorem 2.2]{MR4345837} we find $u^k \in \cH_{p,0}^{\alpha,1}(\bR^d_T)$ satisfying
$$
- \partial_t^\alpha u^k + D_i \left(a^{ij} D_j u^k \right) = D_i g_i + f^k
$$
in $\bR^d_T$.
For $\phi_i, \psi \in C_0^\infty(\bR^d_T)$, $i=1,\ldots,d$, using \cite[Theorem 2.2]{MR4345837} again, find $w \in \cH_{p',0}^{\alpha,1}\left((-T,0) \times \bR^d\right)$, $1/p + 1/p' = 1$, satisfying
$$
- \partial_t^{\alpha} w + D_i \left( a^{ji}(-t,x_1) D_j w \right) = D_i \left(-\phi_i(-t,x) \right) + \psi(-t,x)
$$
in $(-T,0) \times \bR^d$, where $\partial_t^\alpha = \partial_t I_{-T}^{1-\alpha}$, with the estimate
\begin{equation}
							\label{eq1113_01}
\|w\|_{\cH_{p',0}^{\alpha,1}\left((-T,0) \times \bR^d\right)} \leq N \|\phi_i\|_{L_{p'}\left((-T,0) \times \bR^d\right)} + N \|\psi\|_{L_{p'}\left((-T,0) \times \bR^d\right)}.
\end{equation}
Since
$$
1 - \frac{d+2/\alpha}{p'} \geq - \frac{d+2/\alpha}{q'},
$$
by \cite[Theorem 7.5]{MR4345837} it follows that
\begin{equation}
							\label{eq1113_02}
\|w\|_{L_{q'}\left((-T,0) \times \bR^d\right)} \leq N\|w\|_{\cH_{p',0}^{\alpha,1}\left((-T,0) \times \bR^d\right)},
\end{equation}
where $N = N(d,\alpha,p,q)$.
Then, by proceeding as in the proof of \cite[Theorem 2.1]{MR4030286} we arrive at
\begin{align*}
&\int_0^T\int_{\bR^d} \left( \phi_i D_i u^k + \psi u^k \right) \, dx \, dt
\\
&= \int_0^T \int_{\bR^d} \left( f^k(t,x) w(-t,x) - g_i(t,x) D_iw(-t,x)\right) \, dx \, dt
\\
&\leq \|f^k\|_{L_q} \|w\|_{L_{q'}} + \|g_i\|_{L_p} \|Dw\|_{L_{p'}},
\end{align*}
where
$L_q, L_p = L_q, L_p(\bR^d_T)$ and $L_{q'}, L_{p'} = L_{q'}, L_{p'}((-T,0) \times \bR^d)$.
Combining the above inequality with \eqref{eq1113_01} and \eqref{eq1113_02} shows that
$$
\|u^k\|_{L_p(\bR^d_T)} + \|D u^k\|_{L_p(\bR^d_T)} \leq N \|g_i\|_{L_p(\bR^d_T)} + N\|f^k\|_{L_q(\bR^d_T)}.
$$
By this inequality along with the fact that $\|I^{1-\alpha}_0 u^k \|_{L_p} \leq N \|u^k\|_{L_p}$ (see \cite[Remark A.3]{MR3899965}) and $f_k \to f$ in $L_q(\bR^d_T)$, we see that there exists a function $u$ on $\bR^d_T$ such that $u, Du \in L_p(\bR^d_T)$ and $u$ satisfies \eqref{eq1113_03} as well as \eqref{eq1113_04}.
The proposition is proved.
\end{proof}

\begin{remark}
							\label{rem1117_1}
In Proposition \ref{prop1113_1} if $g_i \in L_q \cap L_p (\bR^d_T)$, then we also have $u \in \cH_{q,0}^{\alpha,1}(\bR^d_T)$ by \cite[Theorem 2.2]{MR4345837} with $q$.
In particular, $\partial_t^\alpha u \in \bH_q^{-1}(\bR^d_T)$, but we do not necessarily have $\partial_t^\alpha u \in \bH_p^{-1}(\bR^d_T)$ because $f \in L_q(\bR^d_T)$.
\end{remark}

The proof of Lemma \ref{lem1113_1} below employs a similar iteration argument as in the proof of \cite[Lemma 4.2]{MR2764911} for parabolic equations with the local time derivative.
However, in each iteration step presented here we verify that the solution $v$ belongs to $\cH_{p_j,0}^{\alpha,1}$, as we need to apply an embedding result that holds only in this space.
Specifically, if we follow the proof of \cite[Lemma 4.2]{MR2764911}, we know that $D_{x'}u, u \in L_{p_1}$, but $\|D_1u\|_{p_1}$ can be controlled by $\|D_{x'}u\|_{p_1}$ only when $u \in \cH_{p_1,0}^{\alpha,1}$ is a priori known.

Recall that $v \in \cH_{p,0}^{\alpha,1}\left((0,T) \times B_R\right)$ is said to satisfy
\[
-\partial_t^\alpha v + D_i \left(a^{ij} D_j v \right) = D_i g_i + f
\]
in $(0,T) \times B_R$ if
\begin{equation*}
\int_0^T\int_{B_R}\left( I^{1-\alpha} v \, \varphi_t \, dx \, dt - a^{ij}D_jv D_i \varphi \right) \, dx \, dt = \int_0^T\int_{B_R} (f \varphi - g_i D_i \varphi) \, dx \, dt
\end{equation*}
for any $\varphi \in C_0^\infty\left([0,T) \times B_R\right)$.
In particular, $\varphi(0,x)$ is not necessarily zero.

\begin{lemma}
							\label{lem1113_1}
Let $\alpha \in (0,1)$, $p_0 \in (1,\infty)$, $T \in (0,\infty)$, $0<r<R<\infty$, and $a^{ij}$ satisfy Assumption \ref{assum0808_2}.
Suppose that $v \in \cH_{p_0,0}^{\alpha,1}\left((0,T) \times B_R\right)$ satisfies
\begin{equation}
							\label{eq1113_05}
-\partial_t^\alpha v + D_i\left(a^{ij}D_j v\right) = 0
\end{equation}
in $(0,T) \times B_R$.
Then, $v \in \cH_{p_1,0}^{\alpha,1}\left((0,T) \times B_r\right)$ for any $p_1 \in (1,\infty)$.
Moreover, for any multi-index $\beta = (\beta_2,\ldots,\beta_d)$ of the order $|\beta|=1,2,\ldots$, the function $v_\beta := D_{x'}^{\beta} v$ belongs to $\cH_{p_1,0}^{\alpha,1}\left((0,T) \times B_r\right)$ and satisfies
\begin{equation}
							\label{eq0202_01}
-\partial_t^\alpha v_\beta + D_i(a^{ij}D_jv_\beta) = 0
\end{equation}
in $(0,T) \times B_r$.
Furthermore, for any $t_0 \leq T$, $0 < \nu < \mu$, and a sequence $\{s_k\}$ satisfying \eqref{eq0212_04} in Appendix \ref{appendix1}, we have
\begin{equation}
							\label{eq1113_09}
\begin{aligned}
&\|Dv_\beta\|_{L_{p_1}((t_0-\nu,t_0) \times B_{r/2})}\\
&\leq \frac{N}{r} \|v_\beta\|_{L_{p_1}\left((t_0-\mu,t_0) \times B_r\right)} + N r \frac{\mu^{1-\alpha}}{\mu-\nu} \|v_\beta\|_{L_{p_1}\left((t_0-2\mu, t_0-\nu) \times B_r \right)}
\\
&\quad + N_1 r \sum_{k=1}^\infty s_k^{-\alpha-1}(s_{k+1}-s_k)^{1-\frac{1}{p_1}} \mu^{\frac{1}{p_1}} \left( \int_{t_0-s_{k+1}}^{t_0-s_k} \int_{B_r}|v_\beta(t,x)|^{p_1} \, dx \, dt\right)^{1/p_1},
\end{aligned}
\end{equation}
where $N = N(d,\delta,\alpha,p_1)$, $N_1=N(d,\delta,\alpha,p_1,N_0)$ (recall $N_0$ from \eqref{eq0212_04}), and $v_\beta$ denotes the zero extension of $v_\beta$ for $t \leq 0$.
\end{lemma}

\begin{proof}
To prove the first assertion in the lemma, it suffices to consider $p_1 \in (p_0,\infty)$.
Take $p_1 \in (p_0, \infty)$ such that
\begin{equation}
							\label{eq1113_10}
1 - \frac{d+2/\alpha}{p_0} \geq - \frac{d+2/\alpha}{p_1}.
\end{equation}
Fix $R_1$ and $R_2$ such that $0 < r < R_1 < R_2 < R$.
Since $v \in \cH_{p_0,0}^{\alpha,1}\left((0,T) \times B_R\right)$, by \cite[Corollary 7.6]{MR4345837} we have $v \in L_{p_1}\left((0,T) \times B_{R_1}\right)$ with the inequality
\begin{equation}
							\label{eq0405_01}
\|u\|_{L_{p_1}((0,T) \times B_{R_1})} \leq N \|u\|_{\cH_{p_0,0}^{\alpha,1}((0,T) \times B_{R_2})}.
\end{equation}
Let $\phi(x)$ be an infinitely differentiable function defined on $\bR^d$ such that $\phi(x) = 1$ on $B_r$ and $\phi(x) = 0$ on $\bR^d \setminus B_{R_1}$.
We see that $\phi v$ belongs to $\cH_{p_0,0}^{\alpha,1}(\bR^d_T)$ and satisfies
$$
-\partial_t^\alpha (\phi v) + D_i\left(a^{ij} D_j (\phi v) \right) = D_i \left( a^{ij} v D_j \phi \right) + a^{ij} D_j v D_i \phi
$$
in $\bR^d_T$, where $a^{ij}v D_j \phi \in L_{p_0} \cap L_{p_1}(\bR^d_T)$ and $a^{ij}D_j v D_i \phi \in L_{p_0}(\bR^d_T)$.
Then, by Proposition \ref{prop1113_1} (also see Remark \ref{rem1117_1}) and \eqref{eq0405_01} it follows that $\phi v, D(\phi v) \in L_{p_1}(\bR^d_T)$ with the estimate
\begin{equation}
							\label{eq1113_07}
\begin{aligned}
&\||v| + |Dv|\|_{L_{p_1}((0,T) \times B_r)} \leq \|v D\phi\|_{L_{p_1}(\bR^d_T)} + \|Dv D\phi\|_{L_{p_0}(\bR^d_T)}
\\
&\leq N \|v\|_{L_{p_1}((0,T) \times B_{R_1})} + N \|Dv\|_{L_{p_0}((0,T) \times B_{R_1})} \leq N \|v\|_{\cH_{p_0,0}^{\alpha,1}((0,T) \times B_{R_2})}
\\
& \leq N \|v\|_{L_{p_0}((0,T) \times B_{R_2})} + N\|Dv\|_{L_{p_0}((0,T) \times B_{R_2})} \leq N \|v\|_{L_{p_0}((0,T) \times B_R)},
\end{aligned}
\end{equation}
where
in the forth inequality we used the equation \eqref{eq1113_05} to bound $\|\partial_t^\alpha v\|_{\bH_{p_0}^{-1}((0,T) \times B_{R_2})}$ by $\|Dv\|_{L_{p_0}((0,T) \times B_{R_2})}$.
The last inequality in \eqref{eq1113_07} follows from \cite[Lemma 4.3]{MR4030286} along with \cite[Theorem 2.2]{MR4345837}.
Hence, from \eqref{eq1113_07} and the equation \eqref{eq1113_05} we obtain that $v \in \cH_{p_1,0}^{\alpha,1}\left((0,T) \times B_r\right)$ and
\begin{equation*}
\|v\|_{\cH_{p_1,0}^{\alpha,1}\left((0,T) \times B_r\right)} \leq N \|v\|_{L_{p_0}((0,T) \times B_R)}.
\end{equation*}
Indeed, to check that $v \in \cH_{p_1,0}^{\alpha,1}\left((0,T) \times B_r\right)$,
we use the equation \eqref{eq1113_05} to see that $v \in L_{p_1}((0,T) \times B_r)$ satisfies the equality \eqref{eq0601_01} with $g_i$ replaced with $a^{ij}D_jv$ and $f = 0$ for all $\varphi \in C_0^\infty\left([0,T) \times B_r\right)$.

We complete the proof of the first assertion for arbitrary $p_1 \in (p_0,\infty)$ by repeating the above argument finitely many times to reach $p_1$.
In particular, in each step of the iteration it is required that $u \in \cH_{p_j,0}^{\alpha,1}((0,T) \times B_{r_j})$, $p_0 < p_j \leq p_1$, $r < r_j < R$, because the embedding (see \cite[Corollary 7.6]{MR4345837}) is for functions in such spaces.

The second assertion of the lemma is in fact a simplified version of \cite[Lemma 4.1]{MR4345837} with no cut-off function $\eta$ and the zero right-hand side.
To be more precise, we set
\[
\Delta_{\ell,h}v(t,x) = \frac{v(t,x+h e_\ell) - v(t,x)}{h}, \quad 0 < h < \frac{R-r}{2},
\]
where $e_\ell$ is the unit vector in the $x_\ell$-direction.
Using the first assertion proved above, we see that $\Delta_{\ell,h}v(t,x) \in \cH_{p_1,0}^{\alpha,1}\left((0,T) \times B_{R_1}\right)$, where $R_1 \in (0, (R+r)/2)$.
Because $a^{ij}$ are functions of only $(t,x_1)$, we also see that
\[
-\partial_t^\alpha \left(\Delta_{\ell,h}v\right) + D_i \left(a^{ij} D_j\left(\Delta_{\ell,h}v\right) \right) = 0
\]
in $(0,T) \times B_{R_1}$ for $\ell = 2,\ldots,d$.
Then, we obtain the assertion using the properties of $\Delta_{\ell,h}$ and \cite[Theorem 2.2]{MR4345837} as in the proof of \cite[Lemma 4.1]{MR4345837}.

Finally, we prove \eqref{eq1113_09}.
By \cite[Lemma 3.3]{MR4030286} the extended function $v_\beta$, which is zero for $t\leq 0$, belongs to $\cH_{p_1,0}^{\alpha,1}\left((S,t_0) \times B_r\right)$ for any $S \leq 0$.
In addition, $v_\beta$ satisfies \eqref{eq0202_01} in $(S,t_0) \times B_r$, where $\partial_t^\alpha v =  \partial_t I_S^{1-\alpha} v$.
Set $S = \min\{t_0-\mu,0\}$ and take $\eta(t)$ from \eqref{eq0212_07}.
Then by \cite[Lemma 3.4]{MR4030286}, $\eta v_\beta$ is in  $\cH_{p_1,0}^{\alpha,1}\left((t_0-\mu,t_0) \times B_r \right)$ and satisfies
\begin{equation}
							\label{eq0203_01}
-\partial_t^\alpha (\eta v_\beta) + D_i\left(a^{ij} D_j (\eta v_\beta) \right) = \cG_\beta(t,x)
\end{equation}
in $(t_0-\mu, t_0) \times B_r$, where $\cG_\beta$ is defined as in \eqref{eq0212_01} with $u$ replaced with $v_\beta$.
By applying \cite[Lemma 4.3]{MR4030286} along with \cite[Theorem 2.2]{MR4345837} to \eqref{eq0203_01}, we have
\begin{multline}
							\label{eq0213_01}
\|D(\eta v_\beta)\|_{L_{p_1}\left((t_0-\mu,t_0) \times B_{r/2}\right)} \\
\leq \frac{N}{r} \|v_\beta\|_{L_{p_1}\left((t_0-\mu,t_0) \times B_r\right)} + N r \|\cG_\beta\|_{L_{p_1}\left((t_0-\mu,t_0)\times B_r\right)},
\end{multline}
where $N = N(d,\delta,\alpha,p_1)$.
To take care of the $L_{p_1}$-norm of $\cG_\beta$, we use the argument, for instance, in the proof of \cite[Lemma 4.1]{MR4186022}.
For the reader's convenience and later usage, we present some details in Appendix \ref{appendix1}. That is, by Lemma \ref{lem0212_1} with $p=p_1$, $\Omega = B_r$, and $u = v_\beta$, we see that
\begin{align*}
&\|\cG_\beta\|_{L_{p_1}\left((t_0-\mu, t_0) \times B_r \right)} \leq N(\alpha) \frac{\mu^{1-\alpha}}{(\mu-\nu)} \|v_\beta\|_{L_{p_1}\left((t_0-2\mu, t_0-\nu) \times B_r \right)}\\
&\, + N(\alpha,N_0) \sum_{k=1}^\infty s_k^{-\alpha-1}(s_{k+1}-s_k)^{1-1/p_1} \mu^{1/p_1} \left( \int_{t_0-s_{k+1}}^{t_0-s_k} \int_{B_r}|v_\beta(t,x)|^{p_1} \, dx \, dt\right)^{1/p_1}.
\end{align*}
Then, we obtain \eqref{eq1113_09} from this and \eqref{eq0213_01} with the inequality
\[
\|Dv_\beta\|_{L_{p_1}\left((t_0-\nu,t_0) \times B_{r/2}\right)} \leq
\|D(\eta v_\beta)\|_{L_{p_1}\left((t_0-\mu,t_0) \times B_{r/2}\right)}.
\]
The lemma is proved.
\end{proof}

\begin{lemma}
							\label{lem1118_1}
Let $\alpha \in (0,1)$, $p_0 \in (1,\infty)$, $t_0 \in (0,\infty)$, $0<r<R<\infty$, and $a^{ij}$ satisfy Assumption \ref{assum0808_2}.
If $v \in \cH_{p_0,0}^{\alpha,1}\left((0,t_0) \times B_R\right)$ satisfies \eqref{eq1113_05} in $(0,t_0) \times B_R$, then for any $p \in (p_0,\infty)$ and $t_1 \leq t_0$, we have
\begin{equation}
							\label{eq1115_01}
\left(|D_{x'}v|^p\right)_{Q_{r/2}(t_1,0)}^{1/p} \leq N \sum_{j=1}^\infty j^{-(1+\alpha)} \left(|D_{x'}v|^{p_0}\right)^{1/p_0}_{Q_r(t_1-(j-1)r^{2/\alpha},0)},
\end{equation}
where $N = N(d,\delta,\alpha,p_0,p)$ and $v$ (here and below) denotes the zero extension of $v$ for $t \leq 0$.
Moreover,
\begin{equation}
							\label{eq1116_01}
\left[D_{x'}v\right]_{C^{\sigma \alpha/2, \sigma}(Q_{r/2}(t_1,0))} \leq N r^{-\sigma} \sum_{j=1}^\infty j^{-(1+\alpha)} \left(|D_{x'}v|^{p_0}\right)^{1/p_0}_{Q_r(t_1-(j-1)r^{2/\alpha},0)},
\end{equation}
where $\sigma = \sigma(d,\alpha,p_0) \in (0,1)$ and $N=N(d,\delta,\alpha,p_0)$.
If we additionally assume that $a^{ij} = a^{ij}(t)$ for all $i,j=1,\ldots,d$, then the inequalities \eqref{eq1115_01} and \eqref{eq1116_01} hold with $Dv$ replacing $D_{x'}v$ on both sides of the inequalities.
\end{lemma}

\begin{proof}
Due to scaling, it suffices to consider $r=1$.
Since $v \in \cH_{p_0,0}^{\alpha,1}((0,t_0)\times B_R)$, it is easy to see that the extension of $v$ as zero for $t \leq 0$ satisfies
\begin{equation}
							\label{eq1118_01}
- \partial_t^\alpha v + D_i\left(a^{ij}D_j v \right) = 0
\end{equation}
in $(S,t_0) \times B_R$  for any $S \leq 0$, where $\partial_t^\alpha v =  \partial_t I_S^{1-\alpha} v$.
Indeed, the extended $v$ satisfies \eqref{eq0601_01} with $\Omega_T$ replaced with $(S,t_0)\times B_R$, $g_i = a^{ij}D_j v$, and $f = 0$ for any $\varphi \in C_0^\infty([S,t_0) \times B_R)$.
Hence, the extended $v$ belongs to $\cH_{p_0,0}^{\alpha,1}\left((S,t_0) \times B_R\right)$ for any $S \leq 0$.
Moreover, by Lemma \ref{lem1113_1}, we have $v \in \cH_{q,0}^{\alpha,1}\left((S,t_0) \times B_1\right)$ for any $q \in (1,\infty)$.
Take $\eta(t)$ from \eqref{eq0212_07} with $t_0$ replaced with $t_1$ and $\mu = 1$, $\nu = (1/2)^{2/\alpha}$.
Then, by choosing $S \leq t_1 - 1$ and using \cite[Lemma 4.1]{MR4345837} with $t_1$ and $t_1-1$ in places of $T$ and $t_0$, respectively, it follows that $D_\ell(\eta v) \in \cH_{q,0}^{\alpha,1}((t_1-1,t_1) \times B_{3/4})$, $\ell = 2,\ldots,d$, and
\begin{equation}
							\label{eq1117_01}
\begin{aligned}
\|D_{\ell}(\eta v)\|_{\cH_{q,0}^{\alpha,1}((t_1-1,t_1) \times B_{3/4})} &\leq N\|D_{\ell} (\eta v)\|_{L_q(Q_1(t_1,0))} + N\|\cG_\ell\|_{L_q(Q_1(t_1,0))}
\\
&\le N \sum_{j=1}^\infty j^{-(1+\alpha)} \left(|D_\ell v|^q\right)^{1/q}_{Q_1\left(t_1-(j-1),0\right)},
\end{aligned}
\end{equation}
where $N = N(d,\delta,\alpha,q)$ and $\cG_\ell$ is defined as in \eqref{eq0212_01} with $u$ replaced with $D_\ell v$.
In particular, we obtain the second inequality in \eqref{eq1117_01} using Lemma \ref{lem0212_1} with $p=q$, $\Omega = B_1$, $s_k = k$, $u = D_\ell v$, $N_0=2$, and $\mu = 1$, $\nu = (1/2)^{2/\alpha}$.

We now prove \eqref{eq1115_01}.
Find $p_1 \in (p_0,\infty)$ satisfying \eqref{eq1113_10}.
Since $D_\ell(\eta v) \in \cH_{p_0,0}^{\alpha,1}((t_1-1,t_1) \times B_{3/4})$, by \cite[Corollary 7.6]{MR4345837} we have
$$
\begin{aligned}
\|D_\ell v\|_{L_{p_1}(Q_{1/2}(t_1,0))} &\leq \|D_\ell (\eta v)\|_{L_{p_1}((t_1-1,t_1) \times B_{1/2})}
\\
&\leq N \|D_\ell(\eta v)\|_{\cH_{p_0,0}^{\alpha,1}((t_1-1,t_1) \times B_{3/4})},
\end{aligned}
$$
which combined with \eqref{eq1117_01} with $q = p_0$ proves
$$
\left(|D_{x'}v|^{p_1}\right)_{Q_{1/2}(t_1,0)}^{1/p_1} \leq N \sum_{j=1}^\infty j^{-(1+\alpha)} \left(|D_{x'}v|^{p_0}\right)^{1/p_0}_{Q_1\left(t_1-(j-1),0\right)}.
$$
If $p_1 \geq p$, we arrive at \eqref{eq1115_01} with $r=1$. If not, by performing the iteration process as in the proof of \cite[Proposition 4.3]{MR4186022}, we eventually arrive at \eqref{eq1115_01}.

To prove \eqref{eq1116_01}, we first assume that $1 - (d+2/\alpha)/p_0 =:\sigma > 0$.
By the fact that $D_\ell(\eta v) \in \cH_{p_0,0}^{\alpha,1}((t_1-1,t_1) \times B_{3/4})$, \cite[Corollary 7.4]{MR4345837} shows that
\begin{equation*}
\begin{aligned}
\left[D_\ell v\right]_{C^{\sigma \alpha/2, \sigma}(Q_{1/2}(t_1,0))} &\leq \left[D_\ell(\eta v)\right]_{C^{\sigma \alpha/2, \sigma}((t_1 - 1, t_1) \times B_{1/2})}
\\
&\leq N \|D_\ell(\eta v)\|_{\cH_{p_0,0}^{\alpha,1}((t_1-1,t_1) \times B_{3/4})}.
\end{aligned}
\end{equation*}
From this and \eqref{eq1117_01} with $q=p_0$ we obtain \eqref{eq1116_01} with $r=1$.
If $1-(d+2/\alpha)/p_0 \leq 0$, by repeating the above argument, we prove \eqref{eq1116_01} with a sufficiently large $p_1$ replacing $p_0$ on the right-hand side so that $1-(d+2/\alpha)/p_1 > 0$.
Then, the right-hand side of \eqref{eq1116_01} is estimated by that of \eqref{eq1115_01} through the iteration process depicted in the proof of \cite[Proposition 4.3]{MR4186022}.

If $a^{ij} = a^{ij}(t)$, we repeat the above proof using the corresponding assertion in \cite[Lemma 4.1]{MR4345837}.
The lemma is proved.
\end{proof}

Denote
$$
V = \sum_{j=1}^d a^{1j} D_jv.
$$
To deal with equations the coefficients of which satisfy Assumption \ref{assum0808_2} (ii), we need the following lemma for equations with coefficients $a^{ij} = a^{ij}(x_1)$.
To utilize results from \cite{MR4030286}, we further assume that $a^{ij}(x_1)$ are infinitely differentiable with bounded derivatives.
However, this restriction is harmless because the estimates we obtain below are independent of the smoothness of $a^{ij}$ as in \cite{MR4030286}.
Recall that all the coefficients $a^{ij}$ in this paper satisfy the ellipticity condition \eqref{eq0310_01}, so do the coefficients $a^{ij}$ in the lemma below.

\begin{lemma}[$a^{ij}=a^{ij}(x_1)$ case]
							\label{lem0201_1}
Let $\alpha \in (0,1)$, $p_0 \in (1,\infty)$, $t_0 \in (0,\infty)$, $0<r<R<\infty$, and $a^{ij} = a^{ij}(x_1)$ be infinitely differentiable with bounded derivatives.
If $v \in \cH_{p_0,0}^{\alpha,1}\left((0,t_0) \times B_R\right)$ satisfies \eqref{eq1113_05}
in $(0,t_0) \times B_R$, then for any $p \in (p_0,\infty)$ and $t_1 \leq t_0$, we have
\begin{equation}
							\label{eq1119_01}
\left(|Dv|^p\right)_{Q_{r/2}(t_1,0)}^{1/p} \leq N \sum_{j=1}^\infty j^{-(1+\alpha)} \left(|Dv|^{p_0}\right)^{1/p_0}_{Q_r(t_1-(j-1)r^{2/\alpha},0)},
\end{equation}
where $N = N(d,\delta,\alpha,p_0,p)$ and $v$ (here and below) denotes the zero extension of $v$ for $t \leq 0$.
Moreover,
\begin{equation}
							\label{eq1119_02}
\left[V\right]_{C^{\sigma\alpha/2,\sigma}(Q_{r/2}(t_1,0))} \leq N r^{-\sigma} \sum_{j=1}^\infty j^{-(1+\alpha)} \left(|Dv|^{p_0}\right)^{1/p_0}_{Q_r(t_1-(j-1)r^{2/\alpha},0)},
\end{equation}
where $\sigma = \sigma(d,\alpha,p_0) \in (0,1)$ and $N=N(d,\delta,\alpha,p_0)$.
\end{lemma}

\begin{proof}
As in the proof of Lemma \ref{lem1118_1}, we consider $r = 1$.
By the observation made at the beginning of the proof of Lemma \ref{lem1118_1}, $v$ satisfies \eqref{eq1118_01} in $(S,t_0) \times B_R$ for any $S \leq 0$ and $v \in \cH_{q,0}^{\alpha,1}((S,t_0) \times B_1)$ for any $q \in (1,\infty)$.
Then, by \cite[Lemma 4.9]{MR4030286}, for a $q_1 \in (q,\infty]$ satisfying
\begin{equation}
							\label{eq1119_03}
q_1 \geq q + \frac{\alpha}{\alpha d + 1 - \alpha},
\end{equation}
we have
\begin{equation}
							\label{eq1119_04}
\begin{aligned}
\|D(\eta v)\|_{L_{q_1}\left((t_1-1,t_1) \times B_{1/2}\right)} &\leq N \|D(\eta v)\|_{L_q(Q_1(t_1,0))} + N \|\cG\|_{L_q(Q_1(t_1,0))}
\\
&\leq N \sum_{j=1}^\infty j^{-(1+\alpha)} \left( |Dv|^q\right)^{1/q}_{Q_1\left(t_1-(j-1),0\right)},
\end{aligned}
\end{equation}
where $N=N(d,\delta,\alpha,q)$, $\eta(t)$ is from \eqref{eq0212_07} with $t_0$ replaced with $t_1$ and $\mu = 1$, $\nu = (1/2)^{2/\alpha}$, $\cG=(\cG_1,\ldots,\cG_d)$, $\cG_\ell$ is as in \eqref{eq0212_01} with $u$ replaced with $D_\ell v$ for $\ell = 1,2,\ldots,d$, and the second inequality is obtained by Lemma \ref{lem0212_1} with $q$ in place of $p$, $\Omega = B_1$, $s_j = j$, $u = D_\ell v$, $N_0=2$, and $\mu = 1$, $\nu = (1/2)^{2/\alpha}$.

Also note that the inequality (4.34) in the proof of \cite[Lemma 4.9]{MR4030286} ($V$ in (4.34) of \cite{MR4030286} equals to $\eta(t)V$ here) shows that
\begin{equation}
							\label{eq1119_05}
\begin{aligned}
\|\eta V\|_{\bH_{q,0}^{\alpha,1}\left((t_1-1,t_1)\times B_{1/2}\right)} &\leq N \|D(\eta v)\|_{L_q(Q_1(t_1,0))} + N \|\cG_\ell\|_{L_q(Q_1(t_1,0))}
\\
&\leq N \sum_{j=1}^\infty j^{-(1+\alpha)} \left( |D_\ell v|^q\right)^{1/q}_{Q_1\left(t_1-(j-1),0\right)},
\end{aligned}
\end{equation}
where $N = N(d,\delta,\alpha, q)$.
Indeed, even if not clearly articulated in \cite[Lemma 4.9]{MR4030286}, we have
\begin{equation*}
\eta(t) V \in \bH_{q,0}^{\alpha,2}((t_1-1,t_1) \times B_{1/2}) \subset \bH_{q,0}^{\alpha,1}((t_1-1,t_1) \times B_{1/2})
\end{equation*}
because by Lemma 4.7 in \cite{MR4030286}
$D_x(\eta v) \in \bH_{q,0}^{\alpha,2} ((t_1-1,t_1) \times B_{1/2})$
and $a^{ij}$ are infinitely differentiable with bounded derivatives.

To prove \eqref{eq1119_01}, find $p_1 \in (p_0,\infty)$ satisfying \eqref{eq1119_03} with $p_1$ and $p_0$ in places of $q_1$ and $q$, respectively.
Then, the inequalities in \eqref{eq1119_04} with $p_1$ and $p_0$ imply \eqref{eq1119_01} with $r=1$, provided that $p_1 \geq p$.
If $p_1 < p$, we use the iteration process as in the proof of Lemma \ref{lem1118_1} (\cite[Proposition 4.3]{MR4186022}) to obtain \eqref{eq1119_01} for $p$.

To show \eqref{eq1119_02}, we first assume that $1 - (d+1/\alpha)/p_0 := \sigma > 0$.
Since $\eta V \in \bH_{p_0,0}^{\alpha,1}((t_1-1,t_1) \times B_{1/2})$, by the embedding \cite[Corollary A.8]{MR4030286}, we have
$$
\begin{aligned}
\left[V\right]_{C^{\sigma \alpha/2, \sigma}(Q_{1/2}(t_1,0))} &\leq \left[\eta V\right]_{C^{\sigma \alpha/2, \sigma}((t_1 - 1, t_1) \times B_{1/2})}
\\
&\leq N \|\eta V\|_{\bH_{p_0,0}^{\alpha,1}((t_1-1,t_1) \times B_{1/2})},
\end{aligned}
$$
which together with \eqref{eq1119_05} with $q = p_0$ proves \eqref{eq1119_02} with $r=1$.
If $1-(d+1/\alpha)/p_0 \leq 0$, we proceed as in the last part of the proof of Lemma \ref{lem1118_1} by using \eqref{eq1119_01}.
The lemma is proved.
\end{proof}

Denote
\[
\Pi_r = \{(x_1,x') \in \bR^d: -r<x_1<r, x' \in \bR^{d-1}\}.
\]

\begin{lemma}
							\label{lem0211_01}
Let $\alpha \in (0,1)$, $p_0 \in (1,\infty)$, $t_0 \in (0,\infty)$, $r\in (0,\infty)$, and $a^{ij}$ satisfy Assumption \ref{assum0808_2}.
Suppose that $w \in \cH_{p_0,0}^{\alpha,1}((0,t_0) \times \Pi_r)$ satisfies
\begin{equation}
							\label{eq0206_02}
-\partial_t^\alpha w + D_i\left(a^{ij}D_j w\right) = D_i g_i + f
\end{equation}
in $(0,t_0) \times \Pi_r$ with $w = 0$ on $(0,t_0) \times \partial \Pi_r$, where $g_i, f \in L_p\left((0,t_0) \times \Pi_r\right)$.
Then, we have
\begin{equation}
							\label{eq0208_02}
\|w\|_{L_{p_0}\left((0,t_0) \times \Pi_r\right)} \leq N r\|g_i\|_{L_{p_0}\left((0,t_0) \times \Pi_r\right)} + r^2 \|f\|_{L_{p_0}\left((0,t_0) \times \Pi_r\right)},
\end{equation}
where $N = N(d,\delta, p_0)$ for $p_0 \geq 2$ and $N = N(d,\delta,\alpha,p_0)$ for $p_0 \in (1,2)$, but independent of $t_0$.
\end{lemma}

\begin{proof}
Thanks to scaling, we set $r = 1$.

For $p_0 \in [2,\infty)$, by applying $|w|^{p_0-2}w$ as a test function to the equation \eqref{eq0206_02} we have
\begin{multline}
							\label{eq0208_01}
\int_0^{t_0}\int_{\Pi_1} (\partial_t^\alpha w) |w|^{p_0-2}w \, dx \, dt
+ \int_0^{t_0}\int_{\Pi_1} a^{ij} D_j w D_i \left( |w|^{p_0-2} w \right) \, dx \, dt
\\
= \int_0^{t_0}\int_{\Pi_1} g_i D_i \left( |w|^{p_0-2} w \right) \, dx \, dt - \int_0^{t_0}\int_{\Pi_1} f |w|^{p_0-2} w \, dx \, dt.
\end{multline}
In fact, $|w|^{p_0-2}w$ may not be qualified as a test function because, for instance, $|w(t_0,x)|^{p_0-2}w(t_0,x)$ may not be zero.
However, upon considering an infinitely differentiable approximation sequence as in the proof of \cite[Lemma 4.1]{MR4030286}, we obtain the above equality for sufficiently smooth $w$.
As explained in the proof of \cite[Lemma 4.1]{MR4030286}, the first term in \eqref{eq0208_01} is non-negative.
From this observation with the uniform ellipticity condition, it follows that
\begin{align*}
&\int_0^{t_0} \int_{\Pi_1} |w|^{p_0-2}|Dw|^2 \, dx \, dt \leq N \int_0^{t_0} \int_{\Pi_1} a^{ij} D_j w D_i w |w|^{p_0-2} \, dx \, dt\\
&= N \int_0^{t_0}\int_{\Pi_1} a^{ij} D_j w D_i \left( |w|^{p_0-2} w \right) \, dx \, dt\\
&\leq N \int_0^{t_0}\int_{\Pi_1} g_i D_i \left( |w|^{p_0-2} w \right) \, dx \, dt - N\int_0^{t_0}\int_{\Pi_1} f |w|^{p_0-2} w \, dx \, dt,
\end{align*}
where $N=N(d,\delta,p_0)$.
Note that
\begin{align*}
&\int_0^{t_0}\int_{\Pi_1} g_i D_i \left( |w|^{p_0-2} w \right) \, dx \, dt = (p_0-1) \int_0^{t_0}\int_{\Pi_1} g_i |w|^{p_0-2} D_i w \, dx \, dt\\
&\leq \varepsilon_1 \int_0^{t_0}\int_{\Pi_1} |w|^{p_0-2}|Dw|^2 \, dx \, dt + \varepsilon_2 \int_0^{t_0}\int_{\Pi_1} |w|^{p_0} \, dx \, dt + N \int_0^{t_0}\int_{\Pi_1} |g|^{p_0} \, dx \, dt
\end{align*}
for arbitrary $\varepsilon_1, \varepsilon_2 > 0$, where $N=N(\varepsilon_1,\varepsilon_2,p_0)$.
Also note that
\[
\int_0^{t_0} \int_{\Pi_1} f |w|^{p_0-2} w \, dx \, dt \leq \varepsilon_3 \int_0^{t_0}\int_{\Pi_1} |w|^{p_0} \, dx \, dt + N \int_0^{t_0}\int_{\Pi_1} |f|^{p_0} \, dx \, dt
\]
for arbitrary $\varepsilon_3 > 0$, where $N = N(\varepsilon_3)$.
Then, since $w(t,-1,x') = w(t,1,x') = 0$, by the Poincar\'{e} inequality, we notice that
\[
\int_0^{t_0}\int_{\Pi_1} |w|^{p_0} \, dx \, dt = \int_0^{t_0}\int_{\Pi_1} \left(|w|^{p_0/2}\right)^2 \, dx \, dt \leq p_0^2 \int_0^{t_0}\int_{\Pi_1} |w|^{p_0-2} |D_1w|^2 \, dx \, dt.
\]
Combining the above inequalities with sufficiently small $\varepsilon_1, \varepsilon_2, \varepsilon_3 > 0$, we arrive at \eqref{eq0208_02} with $r=1$ for $p_0 \geq 2$.
In addition to \eqref{eq0208_02}, for the duality argument below, we also need
\begin{equation}
							\label{eq0209_01}
\|Dw\|_{L_p\left((0,t_0)\times \Pi_1\right)} \leq N \|g\|_{L_p\left((0,t_0)\times \Pi_1\right)} + N \|f\|_{L_p\left((0,t_0)\times \Pi_1\right)},
\end{equation}
where $N = N(d,\delta,\alpha,p_0)$, but independent of $t_0$. To prove this, we write \eqref{eq0206_02} as
\[
-\partial_t^\alpha w + D_i(a^{ij} D_j w) - w = D_i g_i + f -w.
\]
Then by \cite[Proposition 6.2]{MR4345837} with $\lambda = 1$, we have
\[
\|Dw\|_{L_p\left((0,t_0)\times \Pi_1\right)} \leq N \|g\|_{L_p\left((0,t_0)\times \Pi_1\right)} + N \|f\|_{L_p\left((0,t_0)\times \Pi_1\right)} + N \|w\|_{L_p\left((0,t_0)\times \Pi_1\right)},
\]
where $N = N(d,\delta,\alpha,p_0)$, but independent of $t_0$.
From this together with \eqref{eq0208_02} we obtain \eqref{eq0209_01}.

For $p_0 \in (1,2)$, we use the usual duality argument made possible by the existence result \cite[Proposition 6.2]{MR4345837} for the partially bounded domain $(0,t_0) \times \Pi_1$ and the estimates \eqref{eq0208_02} and \eqref{eq0209_01} for $p_0 \geq 2$.
For $\phi \in C_0^\infty\left((0,t_0) \times \Pi\right)$, by utilizing \cite[Proposition 6.2]{MR4345837} we find $v \in \cH_{p_0',0}^{\alpha,1}\left((0,t_0) \times \Pi\right)$, $1/p_0 + 1/p_0' = 1$, satisfying
\begin{equation}
							\label{eq0211_01}
- \partial_t^\alpha v + D_i\left( a^{ji}(t_0-t) D_j v \right) = \phi_0(t_0-t,x)
\end{equation}
in $(0,t_0) \times \Pi_1$ with the Dirichlet boundary condition
\[
v(t,-1,x') = v(t,1,x') = 0.
\]
Set $\varphi(t,x)= v(t_0-t,x)$ and $\phi(t,x) = w(t_0-t,x)$ to be applied as test functions to \eqref{eq0206_02} and \eqref{eq0211_01}, respectively, to get
\begin{align*}
&\int_0^{t_0}\int_{\Pi_1} \left( f(t,x) v(t_0-t,x) - g_i(t,x)D_iv(t_0-t,x) \right) \, dx \, dt\\
&= \int_0^{t_0}\int_{\Pi_1} \phi_0(t,x) w(t,x) \, dx \, dt.
\end{align*}
This combined with the estimates \eqref{eq0208_02} and \eqref{eq0209_01} for $v$ with $p_0' > 2$ proves \eqref{eq0208_02} for $p_0 \in (1,2)$ when $r=1$.
\end{proof}

\begin{lemma}
							\label{lem0221_1}
Let $\alpha \in (0,1)$, $p_0 \in (1,\infty)$, $t_0 \in (0,\infty)$, $0< r < R < \infty$, and $a^{ij}$ satisfy Assumption \ref{assum0808_2}.
Suppose that $w \in \cH_{p_0,0}^{\alpha,1}\left((0,t_0) \times \Pi_R\right)$ satisfies \eqref{eq0206_02} in $(0,t_0) \times \Pi_R$ with $w = 0$ on $(0,t_0) \times \partial \Pi_R$, where $g_i, f \in L_p\left((0,t_0) \times \Pi_R\right)$.
Then, for any $t_1 \leq t_0$ and $\rho>0$, we have
\begin{multline}
								\label{eq0921_06}
\|Dw\|_{L_{p_0}\left((t_1-\rho,t_1) \times B_r\right)}
\leq N\frac{\theta^{1/p_0} R}{R-r}  \sum_{j=0}^\infty  2^{-\alpha j} \left( \dashint_{\!t_1- s_{j+2}\theta}^{\,\,\,t_1-s_j\theta} \int_{\Pi_R} |g_i|^{p_0} \, dx \, dt \right)^{1/p_0}
\\
+ N\frac{\theta^{1/p_0} R^2}{R-r} \sum_{j=0}^\infty  2^{-\alpha j} \left( \dashint_{\!t_1-s_{j+2}\theta}^{\,\,\,t_1-s_j \theta} \int_{\Pi_R} |f|^{p_0} \, dx \, dt \right)^{1/p_0},
\end{multline}
where $\theta = R^{2/\alpha}$ if $\rho < R^{2/\alpha}$ and $\theta = \rho$ if $\rho \geq R^{2/\alpha}$,
\begin{equation}
							\label{eq0926_01}
s_0 = 0, \quad s_j = 2^j, \quad j = 1,2,\ldots,
\end{equation}
$N=N(d,\delta,\alpha,p_0)$, and $w, g_i, f$ denote the zero extensions of them for $t \leq 0$.
\end{lemma}

\begin{remark}
							\label{rem0930_1}
Note that
\begin{equation*}
\dashint_{(t_1-s_{j+2}\theta, t_1-s_j\theta)} \leq \frac{4}{3}\dashint_{(t_1-s_{j+2}\theta, t_1)}.
\end{equation*}
Thus, one can replace the integrals in \eqref{eq0921_06} with those over $(t_1-s_{j+2}\theta,t_1)$, which can be further replaced with $(t_1-s_j\theta,t_1)$ with another constant $N$.
Hence, for instance, instead of the last summation in \eqref{eq0921_06}, we may have
\[
 \sum_{j=1}^\infty 2^{-\alpha j} \left( \dashint_{\!t_1-s_j\theta}^{\,\,\,t_1} \int_{\Pi_R} |f|^{p_0} \, dx \, dt \right)^{1/p_0}.
\]
We employ such a replacement throughout the paper whenever the replacement is necessary or makes the exposition better.
\end{remark}

\begin{proof}[Proof of Lemma \ref{lem0221_1}]
By scaling, we may assume that $0<r < R = 1$.
We further assume that $\rho \geq 1$, which will be removed later.
Note that, for any $S \leq 0$, we have  $w \in \cH_{p_0,0}^{\alpha,1}\left((S,t_1) \times \Pi_1\right)$ satisfying \eqref{eq0206_02} in $(S, t_1) \times \Pi_1$, where $\partial_t^\alpha = \partial_t I_S^{1-\alpha}$.
Thus, by taking $\eta$ from \eqref{eq0212_07} with $t_0$ replaced with $t_1$ and $\mu = 2\rho$, $\nu = \rho$, we see that
$\eta w \in \cH_{p_0,0}^{\alpha,1}\left((t_1-2\rho,t_1) \times \Pi_1 \right)$ satisfies
\[
-\partial_t^{\alpha} (\eta w) + D_i\left(a^{ij}D_j (\eta w) \right) = D_i(\eta g_i) + \eta f + F
\]
in $(t_1-2\rho,t_1) \times \Pi_1$ and also in $(t_1-2\rho,t_1) \times B_1$, where $\partial_t^\alpha = \partial_t I_{t_1-2\rho}^{1-\alpha}$ and $F(t,x)$ is defined as in \eqref{eq0212_01} with $u$ replaced with $w$.
By \cite[Proposition 6.2]{MR4345837} (also see the proof of \cite[Lemma 4.3]{MR4030286}), it follows that
\begin{equation}
							\label{eq0920_04}
\begin{aligned}
&\|Dw\|_{L_{p_0}\left((t_1-\rho,t_1)\times B_r\right)} \leq \|D(\eta w)\|_{L_{p_0}((t_1-2\rho,t_1) \times B_r)}\\
&\leq \frac{N}{1-r} \|w\|_{L_{p_0}((t_1-2\rho,t_1) \times \Pi_1)} + N \|g_i\|_{L_{p_0}((t_1-2\rho,t_1) \times \Pi_1)}\\
&\quad + N(1-r) \||f| + |F|\|_{L_{p_0}\left((t_1-2\rho,t_1) \times \Pi_1\right)},
\end{aligned}
\end{equation}
where $N = N(d,\delta,\alpha,p_0)$.
Set
\begin{equation}
							\label{eq0922_01}
\tau_0 = 0, \quad \tau_j = 2^j\rho, \quad j = 1,2,\ldots.
\end{equation}
Note that the sequence $\{\tau_j\}_{k=1}^\infty$  satisfies \eqref{eq0212_04}.
Then, by Lemma \ref{lem0212_1} with $p = p_0$, $\Omega = \Pi_1$, $u = w$, $N_0 = 2$, $s_j=\tau_j$ from \eqref{eq0922_01}, $\mu=2\rho$, and $\nu = \rho$,
\begin{multline}
							\label{eq0920_05}
\|F\|_{L_{p_0}((t_1-2\rho,t_1) \times \Pi_1)} \leq N \rho^{-\alpha} \left( \int_{t_1-4 \rho}^{t_1-\rho} \int_{\Pi_1} |w(t,x)|^{p_0} \, dx \, dt \right)^{1/p_0}
\\
+ N \rho^{-\alpha + 1/p_0}\sum_{j=1}^\infty 2^{-\alpha j} \left( \dashint_{\!t_1-\tau_{j+1}}^{\,\,\,t_1-\tau_j} \int_{\Pi_1} |w(t,x)|^{p_0} \, dx \, dt \right)^{1/p_0},
\end{multline}
where $N=N(\alpha)$.
Denote
\[
A_j := \left( \dashint_{\!t_1-\tau_{j+1}}^{\,\,\,t_1-\tau_j} \int_{\Pi_1} |w|^{p_0} \, dx \, dt \right)^{1/p_0}
\]
for $j=0,1,2, \ldots$.
By combining \eqref{eq0920_04} and \eqref{eq0920_05}, we have
\begin{align*}
&\|Dw\|_{L_{p_0}\left((t_1-\rho,t_1) \times B_r\right)} \leq N \|g_i\|_{L_{p_0}\left((t_1-2\rho,t_1) \times \Pi_1\right)} + N(1-r) \|f\|_{L_{p_0}\left((t_1-2\rho,t_1) \times \Pi_1\right)}\\
&\quad + \frac{N}{1-r}\left( \int_{t_1-2\rho}^{t_1}\int_{\Pi_1} |w|^{p_0}\right)^{1/p_0} + N(1-r) \rho^{-\alpha}\left(\int_{t_1-4\rho}^{t_1-\rho} \int_{\Pi_1} |w|^{p_0} \right)^{1/p_0}\\
&\quad + N(1-r) \rho^{-\alpha+1/p_0} \sum_{j=1}^\infty 2^{-\alpha j} A_j.
\end{align*}
Note that
\[
\left( \int_{t_1-2\rho}^{t_1}\int_{\Pi_1} |w|^{p_0}\right)^{1/p_0} = (2\rho)^{1/p_0} A_0
\]
and
\begin{align*}
\left(\int_{t_1-4\rho}^{t_1-\rho} \int_{\Pi_1} |w|^{p_0} \right)^{1/p_0} &\leq \left(\int_{t_1-\tau_1}^{t_1} \int_{\Pi_1} |w|^{p_0} \right)^{1/p_0} + \left(\int_{t_1-\tau_2}^{t_1-\tau_1} \int_{\Pi_1} |w|^{p_0} \right)^{1/p_0}\\
&= (2\rho)^{1/p_0}A_0 + (2\rho)^{1/p_0}A_1.
\end{align*}
Hence,
\[
\|Dw\|_{L_{p_0}\left((t_1-\rho,t_1) \times B_r\right)}
\leq N \|g_i\|_{L_{p_0}\left((t_1-2\rho,t_1) \times \Pi_1\right)} + N(1-r) \|f\|_{L_{p_0}\left((t_1-2\rho,t_1) \times \Pi_1\right)}
\]
\[
+ N \rho^{1/p_0} \left( \frac{1}{1-r} + (1-r) \rho^{-\alpha} \right) A_0
+ N(1-r)\rho^{-\alpha + 1/p_0} \sum_{j=1}^\infty 2^{-\alpha j} A_j,
\]
where $N=N(d,\delta,\alpha,p_0)$.
Due to the fact that $r < 1$, $\rho \geq 1$, we have
\begin{multline}
							\label{eq0921_05}
\|Dw\|_{L_{p_0}\left((t_1-\rho,t_1) \times B_r\right)} \leq N \|g_i\|_{L_{p_0}\left((t_1-2\rho,t_1) \times \Pi_1\right)} + N\|f\|_{L_{p_0}\left((t_1-2\rho,t_1) \times \Pi_1\right)}
\\
+ N \rho^{1/p_0}\frac{1}{1-r} A_0 + N\rho^{-\alpha/2+1/p_0} \sum_{j=1}^\infty 2^{-\alpha j} A_j.
\end{multline}

To estimate $A_j$, for each $j=0,1,2,\ldots$, we take $\eta_j$ as $\eta$ in \eqref{eq0212_07} with $t_0= t_1-\tau_j$, $\mu = \tau_{j+2} - \tau_j$, and $\nu = \tau_{j+1}-\tau_j$.
That is,
\[
\eta_j(t) = \left\{
\begin{aligned}
1 \quad &\text{if} \quad t \geq t_1-\tau_{j+1},
\\
0 \quad &\text{if} \quad t \leq t_1-\tau_{j+2},
\end{aligned}
\right.
\]
\[
|\eta_j'(t)| \leq \frac{2}{\tau_{j+2}-\tau_{j+1}} = 2^{-j}\rho^{-1}, \quad j = 0,1,2,\ldots.
\]
Then, $\eta_j w \in \cH_{p_0,0}^{\alpha,1}\left((t_1-\tau_{j+2}, t_1-\tau_j) \times \Pi_1 \right)$ satisfies
\begin{equation}
							\label{eq0921_01}
-\partial_t^\alpha(\eta_j w) + D_i\left(a^{ij} D_j (\eta_j w) \right) = D_i(\eta_j g_i) + \eta_j f + F_j
\end{equation}
in $(t_1-\tau_{j+2}, t_1-\tau_j) \times \Pi_1$,
where $F_j(t,x)$ is as in \eqref{eq0212_01} with $u$ and $\eta$ replaced with $w$ and $\eta_j$, respectively.
By Lemma \ref{lem0211_01} applied to \eqref{eq0921_01} it follows that
\begin{multline}
							\label{eq0921_02}
\|w\|_{L_{p_0}\left((t_1-\tau_{j+1}, t_1-\tau_j) \times \Pi_1\right)} \leq \|(\eta_j w)\|_{L_{p_0}\left((t_1-\tau_{j+2}, t_1-\tau_j) \times \Pi_1\right)}
\\
\leq N \||g_i|+|f|+|F_j|\|_{L_{p_0}\left((t_1-\tau_{j+2}, t_1-\tau_j) \times \Pi_1\right)},
\end{multline}
where $N = N(d,\delta, \alpha, p_0)$.
For $F_j$ above, we set
\[
\tilde{s}_k = \tau_{j+k+1}-\tau_j, \quad k=1,2,\ldots,
\]
and use Lemma \ref{lem0212_1} with $\{\tilde{s}_k\}$.
That is, by Lemma \ref{lem0212_1} with $p=p_0$, $t_0 = t_1-\tau_j$, $\Omega = \Pi_1$, $\tilde{s}_k$, $u=w$, $N_0=3$, and $\mu=\tau_{j+2}-\tau_j$, $\nu=\tau_{j+1}-\tau_j$, we obtain that
\[
\|F_j\|_{L_{p_0}\left((t_1-\tau_{j+2},t_1-\tau_j) \times \Pi_1\right)} \leq N \rho^{-\alpha} 2^{-\alpha j} \left(\int_{t_1-\tau_j-2(\tau_{j+2}-\tau_j)}^{t_1-\tau_j-(\tau_{j+1}-\tau_j)} \int_{\Pi_1}|w|^{p_0} \, dx \, dt\right)^{1/p_0}
\]
\[
+ N \rho^{-\alpha+1/p_0}2^{j/p_0} \sum_{k=1}^\infty 2^{-\alpha(j+k)} \left( \dashint_{\!t_1 - \tau_{j+k+2}}^{\,\,\,t_1-\tau_{j+k+1}} \int_{\Pi_1} |w(s,x)|^{p_0} \, dx \, ds \right)^{1/p_0}
\]
\[
\leq N \rho^{-\alpha + 1/p_0} 2^{j/p_0} \left(2^{-\alpha j}A_{j+1}+2^{-\alpha j}A_{j+2} + \sum_{k=1}^\infty 2^{-\alpha(j+k)} A_{j+k+1}\right),
\]
where $N=N(\alpha)$ and the right-hand side is independent of $A_j$.
See Remark \ref{rem0214_1}.
From the above inequality and \eqref{eq0921_02} we have
\begin{equation}
							\label{eq0921_03}
A_j \leq N G_j
+ N \rho^{-\alpha} \sum_{k=j+1}^\infty 2^{-\alpha k} A_k
\end{equation}
for $j=0,1,\ldots$, where
\[
G_j = \left( \dashint_{\!t_1-\tau_{j+2}}^{\,\,\,t_1-\tau_j} \int_{\Pi_1} \left(|g_i|^{p_0} + |f|^{p_0}\right) \, dx \, dt \right)^{1/p_0}.
\]
We then fix a positive integer $k_0$ depending only on $d$, $\delta$, $\alpha$, and $p_0$ such that
\[
N \frac{2^\alpha}{2^\alpha-1}2^{-\alpha k_0} < 1/2,
\]
where $N$ is the constant in front of the summation in \eqref{eq0921_03}.
By multiplying both sides of \eqref{eq0921_03} by $2^{-\alpha j}$ and summing for $j=k_0,k_0+1,\ldots$, we have
\begin{equation}
							\label{eq0921_04}
\sum_{j=k_0}^\infty 2^{-\alpha j} A_j \leq N \sum_{j=k_0}^\infty 2^{-\alpha j} G_j
+ N \rho^{-\alpha} \sum_{j=k_0}^\infty 2^{-\alpha j} \sum_{k=j+1}^\infty 2^{-\alpha k} A_k.
\end{equation}
By the choice of $k_0$, one can bound the last double summation in \eqref{eq0921_04} as
\[
N \sum_{k=k_0+1}^\infty  2^{-\alpha k} A_k \sum_{j=k_0}^{k-1} 2^{-\alpha j} \leq N\frac{2^\alpha 2^{-\alpha k_0}}{2^\alpha-1}\sum_{k=k_0+1}^\infty 2^{-\alpha k} A_k < \frac{1}{2} \sum_{k=k_0+1}^\infty 2^{-\alpha k} A_k,
\]
from which, \eqref{eq0921_04}, and the fact that $\rho \geq 1$, it follows that
\[
\sum_{j=k_0}^\infty 2^{-\alpha j} A_j \leq N \sum_{j=k_0}^\infty 2^{-\alpha j} G_j,
\]
where $(N,k_0) = (N,k_0)(d,\delta,\alpha,p_0)$.
For $A_j$ with $j = 0,1,\ldots,k_0-1$, we use the above estimate as well as \eqref{eq0921_03} with induction so that we have
\[
\sum_{j=0}^\infty 2^{-\alpha j} A_j \leq  N \sum_{j=0}^\infty  2^{-\alpha j} G_j,
\]
where $N = N(d,\delta,\alpha, p_0)$.
From this and \eqref{eq0921_05}, we arrive at
\[
\|Dw\|_{L_{p_0}((t_1-\rho,t_1)\times B_r)} \leq N\frac{\rho^{1/p_0}}{1-r} \sum_{j=0}^\infty 2^{-\alpha j} G_j.
\]
Then, for $\rho \geq R^{2/\alpha}$, by using scaling, we obtain  \eqref{eq0921_06}.
For $\rho < R^{2/\alpha}$, we see that
\[
\|Dw\|_{L_{p_0}\left((t_1-\rho,t_1) \times B_r\right)} \leq \|Dw\|_{L_{p_0}\left((t_1-R^{2/\alpha},t_1) \times B_r\right)},
\]
the right-hand side of which is bounded by that of \eqref{eq0921_06} thanks to the case $\rho \geq R^{2/\alpha}$ proved above with $\theta = R^{2/\alpha}$.
This finishes the proof.
\end{proof}

In the estimate \eqref{eq0206_01} below it is essential that no $D_1 v$ appears on the right-hand side of the inequality.

\begin{lemma}
							\label{lem0228_1}
Let $\alpha \in (0,1)$, $p_0 \in (1,\infty)$, $t_0 \in (0,\infty)$, $r, R \in (0,\infty)$ such that $2r<R$, and $a^{ij}$ satisfy Assumption \ref{assum0808_2}.
Also let $a^{11}$ be infinitely differentiable with bounded derivatives if $a^{11}=a^{11}(x_1)$.
Suppose that $v \in \cH_{p_0,0}^{\alpha,1}\left((0,t_0) \times B_R\right)$ satisfies \eqref{eq1113_05} in $(0,t_0)\times B_R$.
Then, there exist $\tilde{v}, \hat{v} \in \cH_{p,0}^{\alpha,1}\left( (0,t_0) \times B_r \right)$ with $v = \tilde{v}+\hat{v}$ in $(0,t_0) \times B_r$ such that, for any $t_1 \leq t_0$, $\tilde{v}$ and $\hat{v}$ satisfy the following.

For $\tilde{v}$, we have
\begin{equation}
							\label{eq0206_01}
\left(|D\tilde{v}|^{p_0}\right)^{1/p_0}_{Q_{r/2}(t_1,0)} \leq N\sum_{k=0}^\infty 2^{-\alpha k} \left(|D_{x'}v|^{p_0}\right)^{1/p_0}_{(t_1-s_{k+1}r^{2/\alpha},t_1-s_kr^{2/\alpha}) \times B_{2r}},
\end{equation}
where $\{s_k\}_{k=0}^\infty$ is the sequence in \eqref{eq0926_01}.

For $\hat{v}$, when $a^{11}=a^{11}(t)$,
\begin{equation}
							\label{eq0222_06}
[D_1\hat{v}]_{C^{\sigma \alpha/2, \sigma}\left(Q_{r/4}(t_1,0)\right)} \leq N r^{-\sigma} \sum_{k=0}^\infty 2^{-\alpha k} \left(|Dv|^{p_0} \right)^{1/p_0}_{(t_1 - 2^kr^{2/\alpha},t_1) \times B_{2r}}
\end{equation}
and, when $a^{11}=a^{11}(x_1)$,
\begin{equation}
							\label{eq0222_07}
[a^{11}(x_1) D_1\hat{v}]_{C^{\sigma \alpha/2, \sigma}(Q_{r/4}(t_1,0))}
\leq N r^{-\sigma} \sum_{k=0}^\infty 2^{-\alpha k} \left(|Dv|^{p_0} \right)^{1/p_0}_{(t_1 - 2^kr^{2/\alpha},t_1) \times B_{2r}}.
\end{equation}
In these statements, $v$, $\tilde{v}$, and $\hat{v}$ denote the zero extension of them for $t \leq 0$, $\sigma = \sigma(d,\alpha,p_0) \in (0,1)$, and $N=N(d,\delta,\alpha,p_0)$.
\end{lemma}

\begin{proof}
As before we only consider $r=1$.
In this case $2 < R$.
By Lemma \ref{lem1113_1} $D_{x'}v \in \cH_{p_0,0}^{\alpha,1}((0,t_0) \times B_2)$.
In particular, $DD_{x'}v \in L_{p_0}((0,t_0) \times B_2)$.
Thus, if we set
\[
g_1 = - \sum_{j=2}^d 1_{B_1} a^{1j}D_jv, \quad f = 1_{B_1} \Delta_{x'}v - \sum_{i=2}^d\sum_{j=1}^d  1_{B_1} a^{ij} D_{ij}v,
\]
where $1_{B_1}$ is an indicator function of $x$,
then $g_1, f \in L_{p_0}\left((0,t_0) \times \Pi_1\right)$ and by \cite[Proposition 6.2]{MR4345837}, there exists a unique $\tilde{v} \in \cH_{p,0}^{\alpha,1}\left( (0, t_0) \times \Pi_1 \right)$ satisfying
\begin{equation}
							\label{eq0222_03}
-\partial_t^\alpha \tilde{v} + D_1(a^{11}D_1\tilde{v}) + \Delta_{x'}\tilde{v} = D_1 g_1 + f
\end{equation}
in $(0,t_0) \times \Pi_1$ with $\tilde{v} = 0$ on $(0,t_0) \times \partial \Pi_1$.

Set $\hat{v} = v - \tilde{v}$, which belongs to $\cH_{p_0,0}^{\alpha,1}\left((0,t_0) \times B_1\right)$.
Since $a^{ij}$ are independent of $x' \in \bR^{d-1}$ and $DD_{x'}v \in L_{p_0}((0,t_0) \times B_2)$,  $v$ satisfies
\[
-\partial_t^\alpha v + D_1(a^{11} D_1 v) + \Delta_{x'}v = D_1 g_1 + f
\]
in $(0,t_0) \times B_1$, which means that $\hat{v}$ satisfies
\begin{equation}
							\label{eq0222_04}
-\partial_t^\alpha \hat{v} + D_1\left(a^{11}D_1 \hat{v}\right) + \Delta_{x'} \hat{v} = 0
\end{equation}
in $(0,t_0) \times B_1$.
These $\tilde{v}$ and $\hat{v}$ are the desired  decomposition of $v$.

We now prove that $\tilde{v}$ and $\hat{v}$ satisfy the inequalities in the lemma.
For each $t_1 \leq t_0$, by Lemma \ref{lem0221_1} with  $(R,r,\rho) = (1,1/2,2^{-2/\alpha})$ applied to $\tilde{v}$ satisfying \eqref{eq0222_03} with the boundary condition $\tilde{v}=0$ on $(0,t_0) \times \partial \Pi_1$, we get
\begin{multline}
							\label{eq0222_01}
\left(|D\tilde{v}|^{p_0}\right)^{1/p_0}_{Q_{1/2}(t_1,0)} \leq N\sum_{j=0}^\infty 2^{-\alpha j} \left( \dashint_{\! t_1-s_{j+2}}^{\,\,\,t_1-s_j} \int_{\Pi_1} |g_1|^{p_0} \, dx \, dt \right)^{1/p_0}
\\
+ N \sum_{j=0}^\infty 2^{-\alpha j} \left( \dashint_{\! t_1-s_{j+2}}^{\,\,\,t_1-s_j} \int_{\Pi_1} |f|^{p_0} \, dx \, dt \right)^{1/p_0},
\end{multline}
where $\{s_j\}_{j=0}^\infty$ is the sequence in \eqref{eq0926_01} and $N = N(d,\delta,\alpha,p_0)$.
Note that, for each $j=0,1,2,\ldots,$
\begin{equation}
							\label{eq0222_02}
\left( \dashint_{\! t_1-s_{j+2}}^{\,\,\,t_1-s_j} \int_{\Pi_1} |g_1|^{p_0} \, dx \, dt \right)^{1/p_0} \leq N\left(|D_{x'}v|^{p_0}\right)^{1/p_0}_{(t_1-s_{j+2}, t_1-s_j) \times B_1}
\end{equation}
and
\begin{equation}
							\label{eq0221_02}
\left( \dashint_{\! t_1-s_{j+2}}^{\,\,\,t_1-s_j} \int_{\Pi_1} |f|^{p_0} \, dx \, dt \right)^{1/p_0} \leq N \left(|DD_{x'}v|^{p_0}\right)^{1/p_0}_{(t_1-s_{j+2}, t_1-s_j) \times B_1},
\end{equation}
where $N=N(d,\delta,p_0)$.
To estimate $DD_{x'}v$ on the right-hand side of \eqref{eq0221_02}, for each $j = 0, 1, \ldots$, we apply Lemma \ref{lem1113_1}, in particular, \eqref{eq1113_09} with
\[
(p_1, t_0, \nu, \mu, r) = (p_0, t_1-s_j, s_{j+2}-s_j, s_{j+3}-s_j, 2)
\]
and $s_k$ in \eqref{eq1113_09} replaced with
\[
\tilde{s}_k = s_{j+k+2}-s_j
\]
so that $N_0 = 3$.
Thus, we have
\begin{align*}
&\|DD_{x'}v\|_{L_{p_0}\left( (t_1-s_{j+2},t_1-s_j) \times B_1\right)} \leq N \|D_{x'}v\|_{L_{p_0}\left((t_1-s_{j+3}, t_1-s_j) \times B_2\right)}\\
&\quad + N 2^{-\alpha j}\|D_{x'}v\|_{L_{p_0}\left((t_1+s_j-2s_{j+3}, t_1-s_{j+2}) \times B_2\right)}\\
&\quad + N \sum_{k=1}^\infty 2^{-\alpha (j+k)-k/p_0} \left( \int_{t_1 - s_{j+k+3}}^{t_1 - s_{j+k+2}} \int_{B_2} |D_{x'}v(t,x)|^{p_0} \, dx \, dt \right)^{1/p_0},
\end{align*}
where $N = N(d,\delta,\alpha, p_0)$. This inequality can be turned into
\[
\left(|DD_{x'}v|^{p_0}\right)^{1/p_0}_{(t_1-s_{j+2},t_1-s_j) \times B_1} \leq N \left(|D_{x'}v|^{p_0}\right)^{1/p_0}_{(t_1-s_{j+3}, t_1-s_j) \times B_2}
\]
\[
+ N \sum_{k=0}^\infty 2^{-\alpha (j+k)} \left( |D_{x'}v|^{p_0} \right)^{1/p_0}_{(t_1 - s_{j+k+3}, t_1 - s_{j+k+2}) \times B_2}.
\]
This together with \eqref{eq0222_01}, \eqref{eq0222_02}, and \eqref{eq0221_02} implies \eqref{eq0206_01} for $r=1$.

For the proof of \eqref{eq0222_06} and \eqref{eq0222_07}, we first consider the case $a^{11}=a^{11}(t)$.
Since $\hat{v}$ satisfies \eqref{eq0222_04}, the coefficients of which are $a^{ij} = a^{ij}(t)$, by Lemma \ref{lem1118_1} (the second assertion of the lemma) with $(r,R) = (1/2,1)$, we have
\begin{multline}
							\label{eq0222_05}
[D_1\hat{v}]_{C^{\sigma \alpha/2, \sigma}(Q_{1/4}(t_1,0))} \leq N \sum_{j=1}^\infty j^{-(1+\alpha)} \left(|D\hat{v}|^{p_0}\right)^{1/p_0}_{Q_{1/2}(t_1-(j-1)2^{-2/\alpha},0)}
\\
= N \sum_{j=1}^\infty j^{-(1+\alpha)} \left(|Dv - D\tilde{v}|^{p_0}\right)^{1/p_0}_{Q_{1/2}(t_1-(j-1)2^{-2/\alpha},0)},
\end{multline}
where $\sigma = \sigma(d,\alpha,p_0) \in (0,1)$ and $N = N(d,\delta, \alpha, p_0)$.
By \eqref{eq0206_01} with $r=1$, the terms involving $D\tilde{v}$ on the right-hand side of \eqref{eq0222_05} are estimated as
\begin{align*}
&\left(|D\tilde{v}|^{p_0}\right)^{1/p_0}_{Q_{1/2}(t_1-(j-1)2^{-2/\alpha},0)}
\\
&\leq N\sum_{k=0}^\infty 2^{-\alpha k} \left(|D_{x'}v|^{p_0}\right)^{1/p_0}_{(t_1-(j-1)2^{-2/\alpha}-s_{k+1},t_1-(j-1)2^{-2/\alpha}-s_k) \times B_2}
\end{align*}
for $j=1,2,\ldots$.
Combining this with \eqref{eq0222_05}, we have
\begin{equation}
							\label{eq0223_01}
[D_1\hat{v}]_{C^{\sigma \alpha/2, \sigma}(Q_{1/4}(t_1,0))} \leq N \sum_{j=1}^\infty j^{-(1+\alpha)} \sum_{k=0}^\infty 2^{-\alpha k} \left(|Dv|^{p_0} \right)^{1/p_0}_{(t_1 - s_{k+1}^j,t_1 - s_k^j) \times B_2},
\end{equation}
where $s_k^j = (j-1)2^{-2/\alpha} + s_k$.
We then proceed as in the proof of \cite[Proposition 4.7]{MR4186022} with slightly different details as follows.
The double summation in \eqref{eq0223_01} equals
\[
\sum_{k=0}^\infty 2^{-\alpha k} \sum_{j=1}^\infty j^{-(1+\alpha)} \left(|Dv|^{p_0} \right)^{1/p_0}_{(t_1 - s_{k+1}^j,t_1 - s_k^j) \times B_2} := I_1 + I_2,
\]
where
\[
I_1 = \sum_{k=0}^\infty  2^{-\alpha k} \sum_{\substack{j \in \bN \\ (j-1)2^{-2/\alpha} < 2^{k+1}}} j^{-(1+\alpha)} \left(|Dv|^{p_0} \right)^{1/p_0}_{(t_1 - s_{k+1}^j,t_1 - s_k^j) \times B_2},
\]
\[
I_2 = \sum_{k=0}^\infty 2^{-\alpha k} \sum_{\substack{j \in \bN \\ (j-1)2^{-2/\alpha} \geq 2^{k+1}}} j^{-(1+\alpha)} \left(|Dv|^{p_0} \right)^{1/p_0}_{(t_1 - s_{k+1}^j,t_1 - s_k^j) \times B_2}.
\]
For each $k=0,1,2,\ldots$, if $j \in \bN$ and $(j-1)2^{-2/\alpha} < 2^{k+1}$, then
\[
(t_1-s_{k+1}^j, t_1-s_k^j) \subset (t_1 - s_{k+2}, t_1),
\]
which implies that
\[
\left(|Dv|^{p_0} \right)^{1/p_0}_{(t_1 - s_{k+1}^j,t_1 - s_k^j) \times B_2} \leq 4^{1/p_0} \left(|Dv|^{p_0} \right)^{1/p_0}_{(t_1 - s_{k+2},t_1) \times B_2}.
\]
Hence, using
\[
\sum_{\substack{j \in \bN \\ (j-1)2^{-2/\alpha} < 2^{k+1}}} j^{-(1+\alpha)} \leq \sum_{j=1}^\infty j^{-(1+\alpha)} < \infty,
\]
we obtain that
\begin{equation}
							\label{eq0226_01}
I_1 \leq N \sum_{k=0}^\infty  2^{-\alpha k} \left(|Dv|^{p_0} \right)^{1/p_0}_{(t_1 - s_{k+2},t_1) \times B_2}.
\end{equation}
To estimate $I_2$, we write
\begin{align*}
&I_2 = \sum_{k=0}^\infty 2^{-\alpha k} \sum_{m=k+1}^\infty \sum_{\substack{j \in \bN \\ 2^m \leq (j-1)2^{-2/\alpha} < 2^{m+1}}} j^{-(1+\alpha)} \left(|Dv|^{p_0} \right)^{1/p_0}_{(t_1 - s_{k+1}^j,t_1 - s_k^j) \times B_2}\\
&\leq N \sum_{k=0}^\infty 2^{-\alpha k} \sum_{m=k+1}^\infty 2^{-m(1+\alpha)} \sum_{\substack{j \in \bN \\ 2^m \leq (j-1)2^{-2/\alpha} < 2^{m+1}}} \left(|Dv|^{p_0} \right)^{1/p_0}_{(t_1 - s_{k+1}^j,t_1 - s_k^j) \times B_2},
\end{align*}
where $N=N(\alpha)$.
By H\"{o}lder's inequality
\begin{align*}
&\sum_{\substack{j \in \bN \\ 2^m \leq (j-1)2^{-2/\alpha} < 2^{m+1}}} \left(|Dv|^{p_0} \right)^{1/p_0}_{(t_1 - s_{k+1}^j,t_1 - s_k^j) \times B_2}\\
&\leq N 2^{m-m/p_0} \left(\sum_{\substack{j \in \bN \\ 2^m \leq (j-1)2^{-2/\alpha} < 2^{m+1}}} \left(|Dv|^{p_0} \right)_{(t_1 - s_{k+1}^j,t_1 - s_k^j) \times B_2}\right)^{1/p_0},
\end{align*}
where $N = N(\alpha, p_0)$.
For each $k = 0,1,2,\ldots$, find
a positive integer $\cN(k)$ such that
\begin{equation*}
2^{2/\alpha} (s_{k+1}-s_k) \leq \cN(k) < 2^{2/\alpha} (s_{k+1}-s_k) + 1.
\end{equation*}

Then, we have
\[
\left(|Dv|^{p_0} \right)_{(t_1 - s_{k+1}^j,t_1 - s_k^j) \times B_2}
\leq 2^{-k}\sum_{\ell=1}^{\cN(k)} \int_{t_1-s_k^j - \ell 2^{-2/\alpha}}^{t_1-s_k^j - (\ell-1)2^{-2/\alpha}} \dashint_{B_2}|Dv|^{p_0}\,dx\,dt,
\]
from which it follows that
\begin{align*}
&\sum_{\substack{j \in \bN \\ 2^m \leq (j-1)2^{-2/\alpha} < 2^{m+1}}} \left(|Dv|^{p_0} \right)_{(t_1 - s_{k+1}^j,t_1 - s_k^j) \times B_2}\\
&\leq 2^{-k}\sum_{\ell=1}^{\cN(k)} \sum_{\substack{j \in \bN \\ 2^m \leq (j-1)2^{-2/\alpha} < 2^{m+1}}} \int_{t_1 -s_k - (j+\ell-1)2^{-2/\alpha}}^{t_1 -s_k - (j+\ell-2)2^{-2/\alpha}} \dashint_{B_2}|Dv|^{p_0}\,dx\,dt\\
&\leq N(\alpha) (2^{k+1}+2^{m+1}+1) \dashint_{\!t_1-2^{k+1}-2^{m+1}-1}^{\,\,\,t_1} \dashint_{B_2}|Dv|^{p_0}\,dx\,dt.
\end{align*}
Hence, using the fact that $m \geq k+1$, we have
\begin{align*}
I_2 &\leq N \sum_{k=0}^\infty 2^{-\alpha k} \sum_{m=k+1}^\infty 2^{-\alpha m} \left(|Dv|^{p_0}\right)^{1/p_0}_{(t_1-2^{m+2}, t_1) \times B_2}
\\
&\leq N \sum_{m=1}^\infty \sum_{k=0}^{m-1} 2^{-\alpha k} 2^{-\alpha m} \left(|Dv|^{p_0}\right)_{(t_1-2^{m+2}, t_1) \times B_2}^{1/p_0}
\\
&\leq N \sum_{m=0}^\infty 2^{-\alpha m} \left(|Dv|^{p_0}\right)_{(t_1-2^{m+2}, t_1) \times B_2}^{1/p_0},
\end{align*}
where $N = N(d,\delta,\alpha,p_0)$.
This together with the estimate \eqref{eq0226_01} for $I_1$ and the inequality \eqref{eq0223_01} proves \eqref{eq0222_06} for $r=1$, where we have $2^k$ instead of $2^{k+2}$. See Remark \ref{rem0930_1}.

When $a^{11}=a^{11}(x_1)$, we obtain \eqref{eq0222_07} by following the same steps as above for the case $a^{11} = a^{11}(t)$.
The only difference is that we use Lemma \ref{lem0201_1} instead of Lemma \ref{lem1118_1}.
Also note that when applying Lemma \ref{lem0201_1} to the equation \eqref{eq0222_04}, we have $V = a^{11}D_1\hat{v}$.
The lemma is proved.
\end{proof}

\section{Mean oscillation estimates}
							\label{sec04}

We are now ready to present mean oscillation  estimates of solutions to equations.

Below by $u \in \cH_{p_0,0,\operatorname{loc}}^{\alpha,1}(\bR^d_T)$ and $u \in \bH_{p_0,0,\operatorname{loc}}^{\alpha,2}(\bR^d_T)$ we mean that, for each $R > 0$, $u \in \cH_{p_0,0}^{\alpha,1}((0,T) \times B_R)$ and $u \in \bH_{p_0,0}^{\alpha,2}((0,T) \times B_R)$, respectively.
We define $L_{p_0, \operatorname{loc}}(\bR^d_T)$ similarly.

We first obtain mean oscillation estimates for solutions to divergence type equations.

\begin{proposition}
							\label{prop0930_01}
Let $p_0 \in (1,\infty)$, $T \in (0,\infty)$, and $a^{ij}$ satisfy Assumption \ref{assum0808_2}.
Also let $a^{11}$ be infinitely differentiable with bounded derivatives if $a^{11} = a^{11}(x_1)$.
Suppose that $u \in \cH_{p_0,0,\operatorname{loc}}^{\alpha,1}(\bR^d_T)$ satisfies
\[
-\partial_t^\alpha u + D_i(a^{ij} D_ju) = D_i g_i
\]
in $\bR^d_T$, where $g_i
\in L_{p, \operatorname{loc}}(\bR^d_T)$.
Then, for any $(t_0,x_0) \in (0,T] \times \bR^d$, $r \in (0,\infty)$, and $\kappa \in (0,1/16)$, we have the following.
\begin{enumerate}
\item For $D_{x'}u$,
\begin{equation}
							\label{eq0920_01}
\begin{aligned}
&\left( |D_{x'}u - (D_{x'}u)_{Q_{\kappa r}(t_0,x_0)}|\right)_{Q_{\kappa r}(t_0,x_0)}\\
&\leq N \kappa^\sigma \sum_{j=0}^\infty 2^{-\alpha j}
\left(|D_{x'}u|^{p_0}\right)^{1/p_0}_{(t_0- 2^jr^{2/\alpha},t_0)\times B_{r/2}(x_0)}\\
&\quad + N \kappa^{-(d + \frac{2}{\alpha})/p_0} \sum_{j=0}^\infty 2^{-\alpha j/2}
\left(|g_i|^{p_0}\right)_{(t_0-2^j r^{2/\alpha},t_0)\times B_r(x_0)}^{1/p_0}.
\end{aligned}
\end{equation}

\item For $D_1u$, when $a^{11}=a^{11}(x_1)$,
\begin{equation}
							\label{eq1006_01}
\begin{aligned}
&\left( |a^{11}D_1u - (a^{11}D_1u)_{Q_{\kappa r}(t_0,0)}|\right)_{Q_{\kappa r}(t_0,x_0)}\\
&\leq N \kappa^\sigma \sum_{j=0}^\infty  2^{-\alpha j} (|Du|^{p_0})_{(t_0-2^j r^{2/\alpha},t_0) \times B_{r/2}(x_0)}^{1/p_0}\\
&\quad + N \kappa^{-(d+\frac{2}{\alpha})/p_0} \sum_{j=0}^\infty 2^{-\alpha j} (|D_{x'}u|^{p_0})_{(t_0-2^j r^{2/\alpha},t_0)\times B_{r/2}(x_0)}^{1/p_0}
\\
&\quad + N \kappa^{-(d + \frac{2}{\alpha})/p_0} \sum_{j=0}^\infty 2^{-\alpha j/2} \left(|g_i|^{p_0}\right)_{(t_0-2^j r^{2/\alpha},t_0) \times B_r(x_0)}^{1/p_0},
\end{aligned}
\end{equation}
and, when $a^{11}=a^{11}(t)$, we have \eqref{eq1006_01} with $a^{11}D_1 u$ replaced with $D_1 u$ on the left-hand side of the inequality.
\end{enumerate}
In these estimates, $\sigma = \sigma(d,\alpha,p_0) \in (0,1)$, $N=N(d,\delta,\alpha,p_0)$, and, as in the previous section, all the functions are extended to be zero for $t \leq 0$.
\end{proposition}

\begin{proof}
Because of translation and dilation, we assume that $x_0 = 0$ and $r=1$.
Since $g_i 1_{B_1}
\in L_{p_0}(\bR^d_T)$, by \cite[Proposition 6.2]{MR4345837}, there exists $w \in \cH_{p_0,0}^{\alpha,1}((0,t_0) \times \Pi_1)$ satisfying
\begin{equation}
							\label{eq0928_01}
-\partial_t^\alpha w + D_i(a^{ij} D_j w) = D_i (g_i 1_{B_1})
\end{equation}
in $(0,t_0) \times \Pi_1$ and $w = 0$ on $(0,t_0) \times \partial \Pi_1$.
Set $v = u - w$, which belongs to $\cH_{p_0,0}^{\alpha,1}((0,t_0) \times B_1)$ and satisfies
\[
-\partial_t^\alpha v + D_i(a^{ij}D_iv) = 0
\]
in $(0,t_0) \times B_1$.

We first prove \eqref{eq0920_01}.
Write
\begin{align*}
&\left( |D_{x'}u - (D_{x'}u)_{Q_{\kappa}(t_0,0)}|\right)_{Q_{\kappa}(t_0,0)}\\
&\leq \left( |D_{x'}v - (D_{x'}v)_{Q_{\kappa}(t_0,0)}|\right)_{Q_{\kappa}(t_0,0)} + 2 \left(|D_{x'}w|\right)_{Q_{\kappa}(t_0,0)} := J_1 + J_2.
\end{align*}
Since $\kappa < 1/16$, it follows that
\[
J_1
\leq 3 \kappa^\sigma [D_{x'}v]_{C^{\sigma \alpha/2,\sigma}(Q_{1/4}(t_0,0))}
\]
for $\sigma \in (0,1)$ in Lemma \ref{lem1118_1}, from which with $(R,r)=(1,1/2)$ we get
\begin{align*}
\left[D_{x'}v\right]_{C^{\sigma \alpha/2,\sigma}(Q_{1/4}(t_0,0))} &\leq N \sum_{j=1}^\infty j^{-(1+\alpha)} \left(|D_{x'}v|^{p_0}\right)_{Q_{1/2}(t_0 - (j-1) 2^{-2/\alpha},0)}^{1/p_0}
\\
&\leq N \sum_{j=0}^\infty 2^{-\alpha j} \left( |D_{x'}v|^{p_0} \right)_{(t_0-2^j,t_0) \times B_{1/2}}^{1/p_0},
\end{align*}
where, for the last inequality, see \cite[Remark 4.4]{MR4186022}.
We then use the fact that $u=w+v$ to get
\begin{align}
\left[D_{x'}v\right]_{C^{\sigma \alpha/2,\sigma}(Q_{1/4}(t_0,0))} \leq & N \sum_{j=0}^\infty 2^{-\alpha j} \left( |D_{x'}u|^{p_0} \right)_{(t_0-2^j,t_0) \times B_{1/2}}^{1/p_0}\nonumber
\\
+ & N \sum_{j=0}^\infty 2^{-\alpha j} \left( |D_{x'}w|^{p_0} \right)_{(t_0-2^j,t_0) \times B_{1/2}}^{1/p_0}. \label{eq0929_01}
\end{align}
For the $D_{x'}w$ terms in \eqref{eq0929_01}, by Lemma \ref{lem0221_1} (and Remark \ref{rem0930_1}) with $(R,r,\rho) = (1,1/2,2^j)$ applied to $w$ satisfying \eqref{eq0928_01}, for each $j=0,1,2,\ldots$, we have
\begin{equation}
							\label{eq1007_02}
\left( |Dw|^{p_0} \right)_{(t_0-2^j,t_0) \times B_{1/2}}^{1/p_0} \leq N \sum_{k=0}^\infty 2^{-\alpha k} \left( |g_i|^{p_0} \right)_{(t_0-2^{k+j+2}, t_0) \times B_1}^{1/p_0},
\end{equation}
which shows that
\[
\sum_{j=0}^\infty 2^{-\alpha j} \left( |D_{x'}w|^{p_0} \right)_{(t_0-2^j,t_0) \times B_{1/2}}^{1/p_0} \leq N \sum_{j=0}^\infty \sum_{k=0}^\infty 2^{-\alpha (j+k)}\left( |g_i|^{p_0} \right)_{(t_0-2^{k+j+2}, t_0) \times B_1}^{1/p_0}
\]
\[
\leq N \sum_{k=0}^\infty 2^{-\alpha k/2}\left( |g_i|^{p_0} \right)_{(t_0-2^k, t_0) \times B_1}^{1/p_0}.
\]

Hence,
\[
J_1 \leq N \kappa^\sigma \sum_{j=0}^\infty 2^{-\alpha j} \left( |D_{x'}u|^{p_0} \right)_{(t_0-2^j,t_0) \times B_{1/2}}^{1/p_0} + N \kappa^\sigma \sum_{k=0}^\infty 2^{-\alpha k/2} \left( |g_i|^{p_0}
\right)_{(t_0 - 2^k,t_0) \times B_1}^{1/p_0}.
\]
To estimate $J_2$, since $w$ satisfies \eqref{eq0928_01} in $(0,t_0) \times \Pi_1$,  we use again Lemma \ref{lem0221_1} (and Remark \ref{rem0930_1}) with $(R,r,\rho) = (1,\kappa,\kappa^{2/\alpha})$ and H\"{o}lder's inequality to get
\begin{equation}
							\label{eq0930_03}
\left(|Dw|\right)_{Q_{\kappa}(t_0,0)} \leq \left(|Dw|^{p_0}\right)_{Q_{\kappa}(t_0,0)}^{1/p_0}
\leq N \kappa^{-(d + \frac{2}{\alpha})/p_0} \sum_{j=0}^\infty 2^{-\alpha j} \left( |g_i|^{p_0}
\right)_{(t_0-2^j,t_0) \times B_1}^{1/p_0}.
\end{equation}
Collecting the estimates for $J_1$ and $J_2$ as well as noting that $\kappa^\sigma \leq  \kappa^{-(d + \frac{2}{\alpha})/p_0}$, we arrive at \eqref{eq0920_01} for $r=1$.

For the mean oscillation estimates for $a^{11}D_1 u$ if $a^{11}=a^{11}(x_1)$ or for $D_1 u$ if $a^{11}=a^{11}(t)$, we first write
\[
v = \tilde{v} + \hat{v}
\]
in $(0,t_0) \times B_{1/4}$, which  is due to Lemma \ref{lem0228_1} with $(R,r)=(1,1/4)$.
Hence, $u = w+\tilde{v}+\hat{v}$ in $(0,t_0) \times B_{1/4}$.

For $a^{11}=a^{11}(x_1)$,
we write
\[
\left( |a^{11}D_1u - (a^{11}D_1u)_{Q_{\kappa}(t_0,0)}|\right)_{Q_{\kappa}(t_0,0)} \leq \left( |a^{11}D_1\hat{v} - (a^{11}D_1\hat{v})_{Q_{\kappa}(t_0,0)}|\right)_{Q_{\kappa}(t_0,0)}
\]
\[
+ N (|D_1\tilde{v}|)_{Q_\kappa(t_0,0)} + N (|D_1w|)_{Q_\kappa(t_0,0)} =: J_3+J_4+J_5.
\]
For $J_3$, as for $J_1$ we have
\[
J_3 \leq 3 \kappa^{\sigma}[a^{11}D_1\hat{v}]_{C^{\sigma \alpha/2,\sigma}(Q_{1/16}(t_0,0))}.
\]
By \eqref{eq0222_07} with $r=1/4$,
\begin{align*}
&\left[a^{11} D_1\hat{v} \right]_{C^{\sigma \alpha/2,\sigma}(Q_{1/16}(t_0,0))} \leq N \sum_{j=0}^\infty 2^{-\alpha j} \left(|Dv|^{p_0}\right)_{(t_0-2^{j-4/\alpha}, t_0) \times B_{1/2}}^{1/p_0}\\
&\leq N \sum_{j=0}^\infty 2^{-\alpha j} \left(|Dv|^{p_0}\right)_{(t_0-2^j, t_0) \times B_{1/2}}^{1/p_0}.
\end{align*}
Using the relation $u=w+v$ in $(0,t_0)\times B_1$, for each $j = 0,1,2,\ldots$,
\[
\left(|Dv|^{p_0}\right)_{(t_0-2^j, t_0) \times B_{1/2}}^{1/p_0}
\leq \left(|Du|^{p_0}\right)_{(t_0-2^j, t_0) \times B_{1/2}}^{1/p_0} + \left(|Dw|^{p_0}\right)_{(t_0-2^j, t_0) \times B_{1/2}}^{1/p_0},
\]
where the last term is estimated as in \eqref{eq1007_02}.
Hence,
\begin{align}
							\label{eq0930_02}
\sum_{j=0}^\infty 2^{-\alpha j} (|Dw|^{p_0})_{(t_0-2^j, t_0) \times B_{1/2}}^{1/p_0}
&\leq N \sum_{j = 0}^\infty 2^{-\alpha j} \sum_{k=0}^\infty 2^{-\alpha k} \left( |g_i|^{p_0}
\right)_{(t_0-2^{k+j+2},t_0) \times B_1}^{1/p_0}\notag
\\
&\leq N \sum_{j=0}^\infty 2^{-\alpha j/2} \left( |g_i|^{p_0}
\right)_{(t_0-2^j,t_0) \times B_1}^{1/p_0}.
\end{align}
From the above inequalities, we see that
\[
J_3 \leq N\kappa^\sigma \sum_{j=0}^\infty 2^{-\alpha j} (|Du|^{p_0})_{(t_0-2^j,t_0) \times B_{1/2}}^{1/p_0} + N\kappa^{\sigma} \sum_{j=0}^\infty 2^{-\alpha j/2} (|g_i|^{p_0})_{(t_0-2^j,t_0)\times B_1}^{1/p_0}.
\]
For $J_4$, we notice that from H\"{o}lder's inequality
\[
(|D_1\tilde{v}|)_{Q_\kappa(t_0,0)} \leq \kappa^{-(d+\frac{2}{\alpha})/p_0}(|D_1\tilde{v}|^{p_0})_{Q_{1/8}(t_0,0)}^{1/p_0},
\]
where by \eqref{eq0206_01} with $r=1/4$, we get
\begin{align*}
(|D\tilde{v}|^{p_0})_{Q_{1/8}(t_0,0)}^{1/p_0} &\leq N \sum_{j=0}^{\infty} 2^{-\alpha j} (|D_{x'}v|^{p_0})_{(t_0-s_{j+1}2^{-4/\alpha}, t_0-s_j 2^{-4/\alpha}) \times B_{1/2}}^{1/p_0}
\\
&\leq N \sum_{j=0}^{\infty} 2^{-\alpha j} (|D_{x'}v|^{p_0})_{(t_0-2^j, t_0) \times B_{1/2}}^{1/p_0}.
\end{align*}
Then, using $u=w+v$,
\begin{align*}
(|D_1\tilde{v}|^{p_0})_{Q_{1/8}(t_0,0)}^{1/p_0} &\leq N\sum_{j=0}^{\infty} 2^{-\alpha j} (|D_{x'}u|^{p_0})_{(t_0-2^j, t_0) \times B_{1/2}}^{1/p_0}
\\
&\quad + N \sum_{j=0}^{\infty} 2^{-\alpha j} (|D_{x'}w|^{p_0})_{(t_0-2^j, t_0) \times B_{1/2}}^{1/p_0}.
\end{align*}
This inequality along with \eqref{eq0930_02} gives
\[
J_4 \leq \kappa^{-(d+\frac{2}{\alpha})/p_0} \sum_{j=0}^\infty  \left[2^{-\alpha j} (|D_{x'}u|^{p_0})_{(t_0-2^j,t_0)\times B_{1/2}}^{1/p_0} + 2^{-\alpha j/2} (|g_i|^{p_0})_{(t_0-2^j,t_0)\times B_1}^{1/p_0}\right].
\]
For $J_5$, we use \eqref{eq0930_03}.
Collecting the estimates for $J_3$, $J_4$, and $J_5$, we arrive at \eqref{eq1006_01}.

For $a^{11}=a^{11}(t)$, we proceed as above with $a^{11}D_1 u$ and $a^{11}D_1 \hat{v}$ replaced with $D_1 u$ and $D_1 \hat{v}$, respectively.
In particular, we use \eqref{eq0222_06} for the H\"{o}lder semi-norm of $D_1 \hat{v}$.
The proposition is proved.
\end{proof}

The next proposition presents mean oscillation estimates for non-divergence type equations, which are derived almost directly from the corresponding ones (Proposition \ref{prop0930_01}) for equations in divergence form.

\begin{proposition}
							\label{prop0930_02}
Let $p_0 \in (1,\infty)$, $T \in (0,\infty)$, and $a^{ij}$ satisfy Assumption \ref{assum0808_2}.
Also let $a^{11}$ be infinitely differentiable with bounded derivatives if $a^{11} = a^{11}(x_1)$.
Suppose that $u \in \bH_{p_0,0,\operatorname{loc}}^{\alpha,2}(\bR^d_T)$ satisfies
\[
-\partial_t^\alpha u + a^{ij} D_{ij}u = f
\]
in $\bR^d_T$, where $f \in L_{p, \operatorname{loc}}(\bR^d_T)$.
Then, for any $(t_0,x_0) \in (0,T] \times \bR^d$, $r \in (0,\infty)$, and $\kappa \in (0,\delta^2/16)$, we have the following.
\begin{enumerate}
\item For $D_{x'}^2u$,
\begin{equation}
							\label{eq1007_01}
\begin{aligned}
&\left( |D_{x'}^2 u - (D_{x'}^2 u)_{Q_{\kappa r}(t_0,x_0)}|\right)_{Q_{\kappa r}(t_0,x_0)}\\
&\leq N \kappa^\sigma \sum_{j=0}^\infty  2^{-\alpha j} \left(|D_{x'}^2 u|^{p_0}\right)^{1/p_0}_{(t_0-2^j r^{2/\alpha},t_0) \times B_{r/2}
(x_0)}
\\
&\quad + N \kappa^{-(d + \frac{2}{\alpha})/p_0} \sum_{j=0}^\infty 2^{-\alpha j/2} \left(|f|^{p_0}\right)_{(t_0-2^j r^{2/\alpha},t_0)\times B_r(x_0)}^{1/p_0},
\end{aligned}
\end{equation}

\item For $D_1D_\ell u$, $\ell = 2,\ldots,d$,
\begin{equation}
							\label{eq0930_05}
\begin{aligned}
&\left( |D_1D_\ell u - (D_1D_\ell u)_{Q_{\kappa r}(t_0,0)}|\right)_{Q_{\kappa r}(t_0,0)}\\
&\leq N \kappa^\sigma \sum_{j=0}^\infty 2^{-\alpha j} (|D D_\ell u|^{p_0})_{(t_0-2^j r^{2/\alpha},t_0) \times B_{r/2}(x_0)}^{1/p_0}
\\
&\quad + N \kappa^{-(d+\frac{2}{\alpha})/p_0} \sum_{j=0}^\infty 2^{-\alpha j} (|D_{x'}^2 u|^{p_0})_{(t_0-2^j r^{2/\alpha},t_0)\times B_{r/2}
(x_0)}^{1/p_0}
\\
&\quad + N \kappa^{-(d + \frac{2}{\alpha})/p_0} \sum_{j=0}^\infty  2^{-\alpha j/2} \left(|f|^{p_0}\right)_{(t_0-2^j r^{2/\alpha},t_0) \times B_r(x_0)}^{1/p_0}.
\end{aligned}
\end{equation}
\end{enumerate}
In these estimates, $\sigma = \sigma(d,\alpha,p_0) \in (0,1)$, $N=N(d,\delta,\alpha,p_0)$, and all the functions are extended to be zero for $t \leq 0$.
\end{proposition}

\begin{proof}
For the case $a^{11}=a^{11}(t)$, we set $\cU_\ell = D_\ell u$ for $\ell = 2,\ldots,d$.
By \cite[Lemma 3.2]{MR4345837} $\cU_\ell \in \cH_{p_0,0,\operatorname{loc}}^{\alpha,1}(\bR^d_T)$.
Moreover, $\cU_\ell$ satisfies the divergence type equation
\begin{equation*}
-\partial_t^\alpha \cU_\ell + D_i\left(\tilde{a}^{ij} D_j \cU_\ell\right) = D_\ell f
\end{equation*}
in $\bR^d_T$, where $\tilde{a}^{ij}$ are defined as
\begin{equation*}
\begin{aligned}
&\tilde{a}^{11} = a^{11}, \quad \tilde{a}^{ij} = a^{ij}, \quad i,j=2,\ldots,d,
\\
&\tilde{a}^{1j} = 0, \quad j = 2, \ldots, d, \quad \tilde{a}^{i1} = a^{1i} + a^{i1}, \quad i = 2,\ldots,d.
\end{aligned}
\end{equation*}
We see that the coefficient matrix $\{\tilde{a}^{ij}\}_{i,j=1,\ldots,d}$ satisfies Assumption \ref{assum0808_2} (i).
Then, by applying Proposition \ref{prop0930_01} to $\cU_\ell$, we get \eqref{eq1007_01} and \eqref{eq0930_05}.

For the case $a^{11} = a^{11}(x_1)$, as in the proof of Proposition \ref{prop0930_01}, we assume that $x_0=0$ and $r=1$.
We then use the following change of variables:
\begin{equation*}
y_1 = \chi(x_1) = \int_0^{x_1} \frac{1}{a^{11}(r)}\,dr, \quad y_i = x_i, \quad i = 2,\ldots,d.
\end{equation*}
From the fact that $\delta \leq a^{11} \leq \delta^{-1}$, we see that the inverse $\chi^{-1}(y_1)$ exists and
\begin{equation}
							\label{eq1008_01}
|\chi(x_1)| \leq \delta^{-1}|x_1|, \quad |\chi^{-1}(y_1)| \leq \delta^{-1}|y_1|.
\end{equation}
Set
\begin{equation*}
\cU_\ell(t,y_1,y') = D_\ell u(t,\chi^{-1}(y_1),y'), \quad \ell = 2,\ldots,d.
\end{equation*}
Then, $\cU_\ell$ belongs to $\cH_{p_0,0,\operatorname{loc}}^{\alpha,1}(\bR^d_T)$ and satisfies
\begin{equation*}
-\partial_t^{\alpha} \cU_\ell + D_i\left(\hat{a}^{ij}D_j \cU_\ell \right) = D_\ell \cF
\end{equation*}
in $\bR^d_T$, where $\cF(t,y_1,y') = f(t,\chi^{-1}(y_1),y')$ and $\hat{a}^{ij}$ are defined as
\begin{equation}
							\label{eq1007_04}
\left\{
\begin{aligned}
&\hat{a}^{11}(y_1) = \frac{1}{a^{11}(\chi^{-1}(y_1))}, \quad \hat{a}^{1j} = 0, & j=2,\ldots,d,
\\
&\hat{a}^{i1}(t,y_1) = \hat{a}^{11}(y_1) \left(a^{1i}(t, \chi^{-1}(y_1)) + a^{i1}(t,\chi^{-1}(y_1))\right), & i=2,\ldots,d,
\\
&\hat{a}^{ij}(t,y_1) = a^{ij}(t,\chi^{-1}(y_1)), & i,j=2,\ldots,d.
\end{aligned}
\right.
\end{equation}
Note that $\{\hat{a}^{ij}\}_{i,j=1,\ldots,d}$ satisfies Assumption \ref{assum0808_2} (ii).
For a constant $C$, using the first inequality in \eqref{eq1008_01}, we have
\begin{align*}
&\left( |D_{x'} D_\ell u - C|\right)_{Q_{\kappa}(t_0,0)}
= N \kappa^{-2/\alpha -d} \int_{t_0-\kappa^{2/\alpha}}^{t_0} \int_{B_\kappa} |D_{x'} \cU_\ell(t,\chi(x_1),x') - C| \, dx \, dt\\
&\leq N \kappa^{-2/\alpha -d} \int_{t_0-\kappa^{2/\alpha}}^{t_0} \int_{|\chi(x_1)|^2+|x'|^2 \leq (\delta^{-1}\kappa)^2} |D_{x'} \cU_\ell(t,\chi(x_1),x') - C| \, dx \, dt\\
&\leq N \kappa^{-2/\alpha-d} \int_{t_0-\kappa^{2/\alpha}}^{t_0} \int_{B_{\delta^{-1}\kappa}} |D_{x'}\cU_\ell(t,y_1,y') - C| \, dy \, dt\\
&\leq N \left(|D_{x'}\cU_\ell - C|\right)_{Q_{\delta^{-1}\kappa}(t_0,0)} = N \left(|D_{x'}\cU_\ell - C|\right)_{Q_{\kappa_1 \delta}(t_0,0)},
\end{align*}
where $N=N(d,\delta)$ and $\kappa_1 := \delta^{-2}\kappa$.
Note that $\kappa_1 < 1/16$ because $\kappa < \delta^2/16$.
Then, by Proposition \ref{prop0930_01} with $\kappa_1$ and $r=\delta$ as well as $C = (D_{x'}\cU_\ell)_{Q_{\kappa_1\delta}(t_0,0)}$, we have
\begin{equation}
							\label{eq1007_03}
\begin{aligned}
&\left( |D_{x'}\cU_\ell - (D_{x'}\cU_\ell)_{Q_{\kappa_1\delta}(t_0,0)}|\right)_{Q_{\kappa_1 \delta}(t_0,0)}\\
&\leq N \kappa_1^\sigma \sum_{j=0}^\infty  2^{-\alpha j} \left(|D_{x'}\cU_\ell|^{p_0}\right)^{1/p_0}_{(t_0- 2^j \delta^{2/\alpha},t_0)\times B_{\delta/2}}\\
&\quad + N \kappa_1^{-(d + \frac{2}{\alpha})/p_0} \sum_{j=0}^\infty  2^{-\alpha j/2}
\left(|\cF|^{p_0}\right)_{(t_0-2^j \delta^{2/\alpha},t_0)\times B_\delta}^{1/p_0},
\end{aligned}
\end{equation}
where we note that, due to the second inequality in \eqref{eq1008_01}, for instance,
\[
\left(|\cF|^{p_0}\right)_{(t_0-2^j \delta^{2/\alpha},t_0)\times B_\delta}^{1/p_0} \leq N \left(|f|^{p_0}\right)_{(t_0-2^j,t_0)\times B_1}^{1/p_0}.
\]
Therefore, from \eqref{eq1007_03} with the observation
\[
\left( |D_{x'}D_\ell u - (D_{x'}D_\ell u)_{Q_{\kappa}(t_0,0)}|\right)_{Q_{\kappa}(t_0,0)} \leq 2 \left( |D_{x'} D_\ell u - C|\right)_{Q_{\kappa}(t_0,0)},
\]
we arrive at \eqref{eq1007_01} for $r=1$.
To obtain \eqref{eq0930_05}, we proceed similarly as above upon noting that
\begin{align*}
&\left( |D_1 D_\ell u - C|\right)_{Q_{\kappa}(t_0,0)}\\
&\leq N \kappa^{-2/\alpha-d} \int_{t_0-\kappa^{2/\alpha}}^{t_0} \int_{B_{\delta^{-1}\kappa}} \left|\frac{1}{a^{11}(\chi^{-1}(y_1))} D_1\cU_\ell(t,y_1,y') - C \right| a^{11}(\chi^{-1}(y_1)) \, dy \, dt\\
&\leq N \left( |\hat{a}^{11}D_1 \cU_\ell - C|\right)_{Q_{\delta^{-1}\kappa}(t_0,0)},
\end{align*}
where $\hat{a}^{11}$ is from \eqref{eq1007_04}.
The proposition is proved.
\end{proof}

\section{Proofs of Theorems \ref{thm1122_1} and \ref{thm1122_2}}
							\label{sec05}

\begin{lemma}
							\label{lem1124_1}
Let $\alpha \in (0,1)$, $T \in (0,\infty)$, $p,q \in (1,\infty)$, $K_1 \geq 1$, $w=w_1(t)w_2(x)$,
where
\[
w_1(t) \in A_p(\bR,dt), \quad w_2(x) \in A_q(\bR^d,dx), \quad [w_1]_{A_p} \leq K_1, \quad [w_2]_{A_q} \leq K_1.
\]
Then, there exist $p_0 = p_0(d,p,q,K_1) \in (1,\infty)$ and $\mu = \mu(d,\alpha, p,q,K_1) \in (1,\infty)$, $1/\mu + 1/\nu = 1$, such that
\begin{equation}
							\label{eq1121_02}
p_0 < p_0 \mu < \min\{p,q\}, \quad \frac{1}{p_0 \nu}<\frac{\alpha}{2},
\end{equation}
and the following holds.
If $u \in \cH_{p,q,w,0}^{\alpha,1}(\bR^d_T)$ has compact support in $Q_{R_0}(t_1,0)$ for $t_1 \in [0,T]$, where $R_0$ is from Assumption \ref{assum0808_1}, and satisfies
\[
-\partial_t^\alpha u + D_i(a^{ij}D_j u) = D_i g_i
\]
in $\bR^d_T$, where $a^{ij}$ satisfy Assumption \ref{assum0808_1} $(\gamma_0)$ and $g_i \in L_{p,q,w}(\bR^d_T)$, then for any $(t_0,x_0) \in (0,T] \times \bR^d$, $r \in (0,\infty)$, $\kappa \in (0,1/16)$, we have the following.
\begin{enumerate}
\item For $D_{x'}u$,
\begin{equation}
							\label{eq1008_02}
\begin{aligned}
&\left( |D_{x'}u - (D_{x'}u)_{Q_{\kappa r}(t_0,x_0)}|\right)_{Q_{\kappa r}(t_0,x_0)}\\
&\leq N \kappa^\sigma \sum_{j=0}^\infty 2^{-\alpha j}\left(|D_{x'}u|^{p_0}\right)^{1/p_0}_{(t_0-2^j r^{2/\alpha},t_0)\times B(x_0)}\\
&\quad + N \kappa^{-(d + \frac{2}{\alpha})/p_0} \gamma_0^{1/(\nu p_0)} \sum_{j=0}^\infty 2^{j\left(\frac{1}{p_0\nu} - \frac{\alpha}{2}\right)} \left(|D_x u|^{\mu p_0}\right)^{1/{(\mu p_0)}}_{(t_0-2^j r^{2/\alpha},t_0)\times B_r(x_0)}\\
&\quad + N \kappa^{-(d + \frac{2}{\alpha})/p_0} \sum_{j=0}^\infty 2^{-\alpha j/2}
\left(|g_i|^{p_0}\right)_{(t_0-2^j r^{2/\alpha},t_0)\times B_r(x_0)}^{1/p_0}.
\end{aligned}
\end{equation}

\item For $D_1u$, there exists a function $U_1$ on $Q_{\kappa r}(t_0,x_0)$ such that
\begin{equation}
                    \label{eq9.14}
N_1(\delta) |D_1u| \leq U_1 \leq N_2(\delta) |D_1u|
\end{equation}
in $Q_{\kappa r}(t_0,x_0)$
and
\begin{equation}
							\label{eq1008_03}
\begin{aligned}
&\left( |U_1 - (U_1)_{Q_{\kappa r}(t_0,0)}|\right)_{Q_{\kappa r}(t_0,x_0)}
\leq N \kappa^\sigma \sum_{j=0}^\infty  2^{-\alpha j} (|Du|^{p_0})_{(t_0-2^j r^{2/\alpha},t_0) \times B_{r/2}(x_0)}^{1/p_0}
\\
&\quad + N \kappa^{-(d+\frac{2}{\alpha})/p_0} \sum_{j=0}^\infty  2^{-\alpha j} (|D_{x'}u|^{p_0})_{(t_0-2^j r^{2/\alpha},t_0)\times B_{r/2}(x_0)}^{1/p_0}
\\
&\quad + N \kappa^{-(d + \frac{2}{\alpha})/p_0} \gamma_0^{1/(\nu p_0)} \sum_{j=0}^\infty  2^{j\left(\frac{1}{p_0\nu} - \frac{\alpha}{2}\right)} \left(|D_x u|^{\mu p_0}\right)^{1/{(\mu p_0)}}_{(t_0-2^j r^{2/\alpha},t_0)\times B_r(x_0)}
\\
&\quad + N \kappa^{-(d + \frac{2}{\alpha})/p_0} \sum_{j=0}^\infty 2^{-\alpha j/2} \left(|g_i|^{p_0}\right)_{(t_0-2^j r^{2/\alpha},t_0) \times B_r(x_0)}^{1/p_0}.
\end{aligned}
\end{equation}
\end{enumerate}
In these statements,  $\sigma = \sigma(d,\alpha,p,q,K_1) \in (0,1)$, $N = N(d,\delta,\alpha,p,q,K_1)$, and all the functions are extended to be zero for $t \leq 0$.
\end{lemma}

\begin{proof}
For the given $w_1 \in A_p(\bR, dt)$ and $w_2 \in A_q(\bR^d,dx)$, using the reverse H\"{o}lder's inequality for $A_p$ weights, we find
\[
\sigma_1 = \sigma_1(p,K_1), \quad \sigma_2 = \sigma_2(d,q,K_1)
\]
such that $p-\sigma_1 > 1$, $q-\sigma_2 > 1$, and
\[
w_1 \in A_{p-\sigma_1}(\bR, dt), \quad w_2 \in A_{q-\sigma_2}(\bR^d,dx).
\]
We then find $p_0 \in (1,\infty)$ such that
\[
p_0 < \frac{p}{p-\sigma_1} \quad \text{and} \quad p_0 < \frac{q}{q-\sigma_2}.
\]
Using the above $\sigma_1$, $\sigma_2$, and $p_0$, we set $\mu \in (1,\infty)$ so that
\[
\frac{1}{\mu} > 1 - \frac{\alpha p_0}{2}, \quad p_0 \mu \leq \frac{p}{p-\sigma_1}, \quad p_0 \mu \leq \frac{q}{q-\sigma_2}.
\]
We see that $p_0$ and $\mu$ satisfy \eqref{eq1121_02}.
Note that
\begin{equation*}
\begin{aligned}
w_1 &\in A_{p-\sigma_1} \subset A_{\frac{p}{p_0 \mu}} \subset A_{\frac{p}{p_0}}(\bR, dt),\\
w_2 &\in A_{q-\sigma_2} \subset A_{\frac{q}{p_0 \mu}} \subset A_{\frac{q}{p_0}}(\bR^d, dt).
\end{aligned}
\end{equation*}
From these inclusions and the fact that $u \in \cH_{p,q,w,0}^{\alpha,1}(\bR^d_T)$ it follows that (see the proof of \cite[Lemma 5.10]{MR3812104})
\[
u \in \cH_{p_0\mu, 0, \operatorname{loc}}^{\alpha,1}(\bR^d_T).
\]

To prove the estimates in the lemma, we now fix $(t_0,x_0) \in \bR^d_T$, $r \in (0,\infty)$, and $\kappa \in (0,1/16)$.
Then, it is enough to consider the case
\begin{equation}
							\label{eq1118_02}
Q_{\kappa r}(t_0,x_0) \cap Q_{R_0}(t_1,0) \neq \emptyset.
\end{equation}
Otherwise, the estimates hold trivially.
In the case of \eqref{eq1118_02}, we have
\begin{equation}
							\label{eq1121_01}
t_0 - (\kappa r)^{2/\alpha} < t_1 \quad \text{and} \quad t_1 - R_0^{2/\alpha} < t_0,
\end{equation}
which imply that $u(t,x) = 0$ if $t \leq t_0 - 2 R_0^{2/\alpha}$, provided that $r < R_0$.
For the fixed $(t_0,x_0)$ and $r \in (0,\infty)$, we define $\bar{a}^{ij}$ which are measurable functions of only $t$, $x_1$, or $(t,x_1)$ as follows.
\begin{enumerate}
\item If $r < R_0$,
\begin{enumerate}
\item for $(i,j) \neq (1,1)$, we set
\[
\bar{a}^{ij}(t,x_1) = \dashint_{B_r'(x_0')} a^{ij}(t,x_1,y') \, dy',
\]
\item for $(i,j) = (1,1)$ with
\begin{enumerate}
\item $a^{11}$ satisfying Assumption \ref{assum0808_1} (2.i) at $(t_0,x_0)$, we set
\[
\bar{a}^{11}(t) = \dashint_{B_r(x_0)} a^{11}(t,y) \, dy,
\]

\item $a^{11}$ satisfying Assumption \ref{assum0808_1} (2.ii) at $(t_0,x_0)$, we set
\[
\bar{a}^{11}(x_1) = \dashint_{Q_r'(t_0,x_0')} a^{11}(s,x_1,y') \, dy' \, ds.
\]
\end{enumerate}
\end{enumerate}

\item If $r \geq R_0$,
\begin{enumerate}
\item for $(i,j) \neq (1,1)$, we set
\[
\bar{a}^{ij}(t,x_1) = \dashint_{B_{R_0}'} a^{ij}(t,x_1,y') \, dy',
\]

\item for $(i,j) = (1,1)$ with
\begin{enumerate}
\item $a^{11}$ satisfying Assumption \ref{assum0808_1} (2.i) at $(t_1,0)$, we set
\[
\bar{a}^{11}(t) = \dashint_{B_{R_0}} a^{11}(t,y) \, dy,
\]

\item $a^{11}$ satisfying Assumption \ref{assum0808_1} (2.ii) at $(t_0,0)$, we set
\[
\bar{a}^{11}(x_1) = \dashint_{Q_{R_0}'(t_1,0)} a^{11}(s,x_1,y') \, dy' \, ds.
\]
\end{enumerate}
\end{enumerate}
\end{enumerate}

Using $\bar{a}^{ij}$ defined above, we write
\[
-\partial_t^\alpha u + D_i (\bar{a}^{ij}D_j u) = D_i \bar{g}_i
\]
in $\bR^d_T$, where
\[
\bar{g}_i = (\bar{a}^{ij} - a^{ij}) D_j u + g_i.
\]
Since $\bar{a}^{ij}$ satisfy Assumption \ref{assum0808_2}, upon replacing $g_i$ with $\bar{g}_i$, by Proposition \ref{prop0930_01} we obtain \eqref{eq0920_01}, \eqref{eq1006_01} when $\bar{a}^{11}=\bar{a}^{11}(x_1)$, and a version of \eqref{eq1006_01} with $a^{11}D_1 u$ replaced with $D_1 u$ on the left-hand side of the inequality when $\bar{a}^{11}=\bar{a}^{11}(t)$.
Regarding the terms involving $\bar{g}_i$, because $u$ has compact support in $Q_{R_0}(t_1,0)$, we have
\[
\sum_{j=0}^\infty 2^{-\alpha j/2} \left(|\bar{g}_i|^{p_0}\right)_{(t_0-2^j r^{2/\alpha},t_0)\times B_r(x_0)}^{1/p_0} \leq \sum_{j=0}^\infty 2^{-\alpha j/2} \left(|g_i|^{p_0}\right)_{(t_0-2^j r^{2/\alpha},t_0)\times B_r(x_0)}^{1/p_0}
\]
\[
+ \sum_{j=0}^\infty 2^{-\alpha j/2} \left( |a^{ij}-\bar{a}^{ij}|^{p_0} |Du|^{p_0}1_{Q_{R_0}(t_1,0)}\right)_{(t_0-2^jr^{2/\alpha} ,t_0) \times B_r(x_0)}^{1/p_0} := J_1 + J_2,
\]
where by H\"older's inequality,
\begin{align*}
&\left( |a^{ij}-\bar{a}^{ij}|^{p_0} |Du|^{p_0}1_{Q_{R_0}(t_1,0)}\right)_{(t_0-2^jr^{2/\alpha} ,t_0) \times B_r(x_0)}^{1/p_0}\\
&\leq \left( |a^{ij}-\bar{a}^{ij}|^{p_0 \nu}1_{Q_{R_0}(t_1,0)}\right)_{(t_0-2^jr^{2/\alpha} ,t_0) \times B_r(x_0)}^{1/(p_0 \nu)} \left(|Du|^{p_0 \mu} \right)_{(t_0-2^jr^{2/\alpha} ,t_0) \times B_r(x_0)}^{1/(p_0 \mu)}.
\end{align*}
Set
\[
J_{2,j}:= \left( |a^{ij}-\bar{a}^{ij}|^{p_0 \nu}1_{Q_{R_0}(t_1,0)}\right)_{(t_0-2^jr^{2/\alpha} ,t_0) \times B_r(x_0)}
\]
for $j= 0,1,\ldots$.
We claim that
\begin{equation}
							\label{eq1122_05}
J_{2,j} \leq N 2^j \gamma_0 \quad j = 0,1,2,\ldots,
\end{equation}
where $N = N(d,\alpha)$.
To see this, we split two cases $r < R_0$ and $r \geq R_0$.
In the latter case, by the definition of $\bar{a}^{ij}$ and the boundedness of $a^{ij}$ by $\delta^{-1}$ it follows that
\[
J_{2,j} \leq N \left( |a^{ij}-\bar{a}^{ij}|\right)_{(t_1-R_0^{2/\alpha} ,t_1) \times B_{R_0}} \leq N \gamma_0
\]
for all $j=0,1,2,\ldots$.

For $r < R_0$, we see that
\begin{equation}
							\label{eq1122_04}
|a^{ij}-\bar{a}^{ij}|1_{Q_{R_0}(t_1,0)} = 0 \quad \text{for}\,\,t < t_0 - 2 R_0^{2/\alpha}
\end{equation}
because by \eqref{eq1121_01}, for such $t$, we have
\[
t < t_0 -2 R_0^{2/\alpha} \leq t_0 - (\kappa r)^{2/\alpha} - R_0^{2/\alpha} < t_1 - R_0^{2/\alpha}.
\]
Using the boundedness of $a^{ij}$ by $\delta^{-1}$ and \eqref{eq1122_04},
\begin{align*}
J_{2,j}
&\leq N(\delta) \left( |a^{ij}-\bar{a}^{ij}|1_{Q_{R_0}(t_1,0)}\right)_{(t_0-2^jr^{2/\alpha} ,t_0) \times B_r(x_0)}
\\
&\leq
\left\{
\begin{aligned}
N \left( |a^{ij}-\bar{a}^{ij}|\right)_{(t_0-2^jr^{2/\alpha} ,t_0) \times B_r(x_0)} \quad \text{if} \quad 2^j r^{2/\alpha} < 2 R_0^{2/\alpha},
\\
N \left( |a^{ij}-\bar{a}^{ij}|\right)_{(t_0-2 R_0^{2/\alpha},t_0) \times B_r(x_0)} \quad \text{if} \quad 2^j r^{2/\alpha} \geq 2 R_0^{2/\alpha}.
\end{aligned}
\right.
\end{align*}
Then, from Remark \ref{rem1122_1}  we see that \eqref{eq1122_05} holds.
Hence,
\begin{align*}
J_2 &\leq N \sum_{j=0}^\infty 2^{-\alpha j/2} J_{2,j}^{1/(p_0\nu)} \left(|Du|^{p_0 \mu} \right)_{(t_0-2^jr^{2/\alpha} ,t_0) \times B_r(x_0)}^{1/(p_0 \mu)}
\\
&
\leq N \gamma_0^{1/(p_0\nu)} \sum_{j=0}^\infty 2^{\left(\frac{1}{p_0\nu} - \frac{\alpha}{2}\right)j}\left(|Du|^{p_0 \mu} \right)_{(t_0-2^jr^{2/\alpha} ,t_0) \times B_r(x_0)}^{1/(p_0 \mu)}.
\end{align*}
By combining this estimate with Proposition \ref{prop0930_01}, we obtain \eqref{eq1008_02} and \eqref{eq1008_03}.
In particular, $U_1 = D_1 u$ if $\bar{a}^{11} = \bar{a}^{11}(t)$ and $U_1 = \bar{a}^{11} D_1 u$ if $\bar{a}^{11} = \bar{a}^{11}(x_1)$.
\end{proof}

Similarly, using Proposition \ref{prop0930_02}, we obtain the following lemma.

\begin{lemma}
							\label{lem1124_2}
Let $\alpha \in (0,1)$, $T \in (0,\infty)$, $p,q \in (1,\infty)$, $K_1 \geq 1$, $w=w_1(t)w_2(x)$,
where
\[
w_1(t) \in A_p(\bR,dt), \quad w_2(x) \in A_q(\bR^d,dx), \quad [w_1]_{A_p} \leq K_1, \quad [w_2]_{A_q} \leq K_1.
\]
Then, there exist $p_0 = p_0(d,p,q,K_1) \in (1,\infty)$ and $\mu = \mu(d,\alpha, p,q,K_1) \in (1,\infty)$, $1/\mu + 1/\nu = 1$, such that
\begin{equation*}
p_0 < p_0 \mu < \min\{p,q\}, \quad \frac{1}{p_0 \nu}<\frac{\alpha}{2},
\end{equation*}
and the following holds.
If $u \in \bH_{p,q,w,0}^{\alpha,2}(\bR^d_T)$ has compact support in $(t_1-R_0^{2/\alpha},t_1) \times B_{R_0}$ for $t_1 \in [0,T]$, where $R_0$ is from Assumption \ref{assum0808_1}, and satisfies
\[
-\partial_t^\alpha u + a^{ij} D_{ij} u = f
\]
in $\bR^d_T$, where $a^{ij}$ satisfy Assumption \ref{assum0808_1} $(\gamma_0)$ and $f \in L_{p,q,w}(\bR^d_T)$, then for any $(t_0,x_0) \in (0,T] \times \bR^d$, $r \in (0,\infty)$, $\kappa \in (0,1/16)$, we have the following.
\begin{enumerate}
\item For $D_{x'}^2 u$,
\begin{align}
							\label{eq5.9}
&\left( |D_{x'}^2 u - (D_{x'}^2 u)_{Q_{\kappa r}(t_0,x_0)}|\right)_{Q_{\kappa r}(t_0,x_0)}
\leq N \kappa^\sigma \sum_{j=0}^\infty 2^{-\alpha j}\left(|D_{x'}^2 u|^{p_0}\right)^{1/p_0}_{(t_0-2^j r^{2/\alpha},t_0)\times B(x_0)}\notag
\\
&\quad + N \kappa^{-(d + \frac{2}{\alpha})/p_0} \gamma_0^{1/(\nu p_0)} \sum_{j=0}^\infty 2^{j\left(\frac{1}{p_0\nu} - \frac{\alpha}{2}\right)} \left(|D^2 u|^{\mu p_0}\right)^{1/{(\mu p_0)}}_{(t_0-2^j r^{2/\alpha},t_0)\times B_r(x_0)}\notag
\\
&\quad + N \kappa^{-(d + \frac{2}{\alpha})/p_0} \sum_{j=0}^\infty 2^{-\alpha j/2}
\left(|f|^{p_0}\right)_{(t_0-2^j r^{2/\alpha},t_0)\times B_r(x_0)}^{1/p_0}.
\end{align}

\item For $D_1D_\ell u$, $\ell = 2,\ldots,d$,
\begin{align}
							\label{eq5.10}
&\left( |D_1D_\ell u - (D_1D_\ell u)_{Q_{\kappa r}(t_0,0)}|\right)_{Q_{\kappa r}(t_0,x_0)}\notag\\
&\leq N \kappa^\sigma \sum_{j=0}^\infty 2^{-\alpha j} (|DD_\ell u|^{p_0})_{(t_0-2^j r^{2/\alpha},t_0) \times B_{r/2}(x_0)}^{1/p_0}\notag
\\
&\quad + N \kappa^{-(d+\frac{2}{\alpha})/p_0} \sum_{j=0}^\infty 2^{-\alpha j} (|D_{x'}^2 u|^{p_0})_{(t_0-2^j r^{2/\alpha},t_0)\times B_{r/2}(x_0)}^{1/p_0}\notag
\\
&\quad + N \kappa^{-(d + \frac{2}{\alpha})/p_0} \gamma_0^{1/(\nu p_0)} \sum_{j=0}^\infty 2^{j\left(\frac{1}{p_0\nu} - \frac{\alpha}{2}\right)} \left(|D^2 u|^{\mu p_0}\right)^{1/{(\mu p_0)}}_{(t_0-2^j r^{2/\alpha},t_0)\times B_r(x_0)}\notag
\\
&\quad + N \kappa^{-(d + \frac{2}{\alpha})/p_0} \sum_{j=0}^\infty 2^{-\alpha j/2} \left(|f|^{p_0}\right)_{(t_0-2^j r^{2/\alpha},t_0) \times B_r(x_0)}^{1/p_0}.
\end{align}
\end{enumerate}
In these statements, $\sigma = \sigma(d,\alpha,p,q,K_1) \in (0,1)$, $N = N(d,\delta,\alpha,p,q,K_1)$, and all the functions are extended to be zero for $t \leq 0$.
\end{lemma}

\begin{remark}
							\label{rem0212_1}
Lemma \ref{lem1124_2} is analogous to \cite[Lemma 5.1]{MR4186022}, where $a^{ij}(t,x)$ are merely measurable in $t$ and have small mean oscillations in $x \in \bR^d$.
However, the inequality (5.2) in \cite{MR4186022} must be expressed with infinite summations on the right-hand side, as in \eqref{eq5.9} and \eqref{eq5.10} instead of the strong maximal functions.
Although the inequality (5.2) in \cite{MR4186022} is correct, using the mean oscillation estimates with a perturbation argument to derive $L_p$-estimates requires an inequality with infinite summations.
The proof of \cite[Lemma 5.1]{MR4186022} actually establishes such an estimate.
\end{remark}

To prove our main theorems, we use the following maximal and strong maximal functions.
For $(t_0,x_0) \in (-\infty,T) \times \bR^d$ with $T \in (-\infty,\infty]$ and a function $f$ defined on $(-\infty,T) \times \bR^d$, we set
$$
\cM f(t_0,x_0) = \sup_{Q_R(t,x) \ni (t_0,x_0)} \dashint_{Q_R(t,x)} |f(s,y)| \, dy \, ds
$$
and
$$
\cS\cM f(t_0,x_0) = \sup_{Q_{R_1,R_2}(t,x) \ni (t_0,x_0)} \dashint_{Q_{R_1,R_2}(t,x)} |f(s,y)| \, dy \, ds,
$$
where the supremum are taken over all $Q_R(t,x)$ and $Q_{R_1,R_2}(t,x)$ such that $(t,x) \in (-\infty,T] \times \bR^d$.
We also use sharp functions defined as follows in the proofs:
\[
f^{\#}(t_0,x_0) = \sup_{Q_R(t,x) \ni (t_0,x_0)} \dashint_{Q_R(t,x)} |f(s,y) - (f)_{Q_R(t,x)}| \, dy \, ds,
\]
where the supremum is taken as above.

\begin{proof}[Proof of Theorem \ref{thm1122_1}]
We first prove
\begin{equation}
							\label{eq1124_06}
\|Du\|_{L_{p,q,w}(\bR^d_T)} \leq N \|g_i\|_{L_{p,q,w}(\bR^d_T)}
\end{equation}
for $u \in \cH_{p,q,w,0}^{\alpha,1}(\bR^d_T)$ having a compact support in $Q_{R_0}(t_1,0)$ and satisfying \eqref{eq1124_02} with
\begin{equation}\label{eq9.47}
a^i = b^i = c = f = 0.
\end{equation}
It follows from Lemma \ref{lem1124_1} (1) that for any $(t,x) \in (-\infty,T) \times \bR^d$ and $\kappa \in (0,1/16)$, we have
\begin{align*}
(D_{x'}u)^\#(t,x)
&\leq N \kappa^\sigma \big(\cS\cM |D_{x'}u|^{p_0}(t,x)\big)^{1/p_0}
\\
&\quad + N \kappa^{-(d + \frac{2}{\alpha})/p_0} \gamma_0^{1/(\nu p_0)} \big(\cS\cM |Du|^{\mu p_0}(t,x)\big)^{1/(\mu p_0)}
\\
&\quad + N \kappa^{-(d + \frac{2}{\alpha})/p_0}  \big(\cS\cM |g_i|^{p_0}(t,x)\big)^{1/p_0},
\end{align*}
where all the functions are extended as zero for $t \leq 0$.
Then by the weighted mixed-norm Hardy-Littlewood maximal function theorem (see, for instance, \cite[Theorem 5.2]{MR4186022}) and the weighted mixed-norm Fefferman-Stein sharp function theorem (see \cite[Corollary 2.7 and (2.4)]{MR3812104}), we get
\begin{equation}
							\label{eq0210_01}
\|D_{x'}u\|
\leq N \kappa^\sigma \|D_{x'}u\|
+ N \kappa^{-(d + \frac{2}{\alpha})/p_0} \gamma_0^{1/(\nu p_0)} \|Du\|+ N \kappa^{-(d + \frac{2}{\alpha})/p_0}  \|g_i\|,
\end{equation}
where $\|\cdot\|=\|\cdot\|_{L_{p,q,w}(\bR^d_T)}$ and the constant $N$ is independent of $\kappa$.
To estimate $D_1u$, we see that Lemma \ref{lem1124_1}  (2) implies the following.
For each $Q_{\kappa r}(t_0,x_0)$, where $(t_0,x_0) \in (0,T] \times \bR^d$, $r \in (0,\infty)$, and $\kappa \in (0,1/16)$, there exists a function $U_1$ on $Q_{\kappa r}(t_0,x_0)$ such that \eqref{eq9.14} holds and
\begin{align*}
&\left( |U_1 - (U_1)_{Q_{\kappa r}(t_0,x_0)}|\right)_{Q_{\kappa r}(t_0,x_0)}
\leq N \kappa^\sigma \big(\cS\cM |Du|^{p_0}(t,x)\big)^{1/p_0}
\\
&\quad + N \kappa^{-(d+\frac{2}{\alpha})/p_0} \big(\cS\cM |D_{x'}u|^{p_0}(t,x)\big)^{1/p_0}
\\
&\quad + N \kappa^{-(d + \frac{2}{\alpha})/p_0} \gamma_0^{1/(\nu p_0)} \big(\cS\cM |Du|^{\mu p_0}(t,x)\big)^{1/(\mu p_0)}
\\
&\quad + N \kappa^{-(d + \frac{2}{\alpha})/p_0} \big(\cS\cM |g_i|^{p_0}(t,x)\big)^{1/p_0}
\end{align*}
for all $(t,x) \in Q_{\kappa r}(t_0,x_0)$.
We now use the weighted mixed-norm Hardy-Littlewood maximal function theorem as above and \cite[Corollary 2.8]{MR3812104} along with the inequality \eqref{eq9.14} to get
\begin{equation}
							\label{eq0210_02}
\begin{aligned}
\|D_{1}u\|
&\leq N \kappa^\sigma \|Du\|+N \kappa^{-(d + \frac{2}{\alpha})/p_0}\|D_{x'}u\|\\
&\quad + N \kappa^{-(d + \frac{2}{\alpha})/p_0} \gamma_0^{1/(\nu p_0)} \|Du\|+ N \kappa^{-(d + \frac{2}{\alpha})/p_0}  \|g_i\|,
\end{aligned}
\end{equation}
where, again, $\|\cdot\|=\|\cdot\|_{L_{p,q,w}(\bR^d_T)}$ and the constant $N$ is independent of $\kappa$.
Combining \eqref{eq0210_01} and \eqref{eq0210_02} gives
\begin{align*}
&\|D_{x'}u\|+ (2N)^{-1}\kappa^{(d + \frac{2}{\alpha})/p_0}\|D_{1}u\|
\le  (N \kappa^\sigma+1/2) \|D_{x'}u\|\\
&\,\,+N(\kappa^{-(d + \frac{2}{\alpha})/p_0} \gamma_0^{1/(\nu p_0)}
+\kappa^{\sigma+(d + \frac{2}{\alpha})/p_0})\|Du\|+ N \kappa^{-(d + \frac{2}{\alpha})/p_0}  \|g_i\|.
\end{align*}
By first choosing $\kappa$ sufficiently small and then $\gamma_0$ small, we absorb the first two terms on the right-hand side above and reach \eqref{eq1124_06}.

Now we remove the small support condition on $u$ and \eqref{eq9.47} as in the proofs of \cite[Lemmas 6.4 and 6.5]{MR4030286}. By using a partition of unity argument and S. Agmon's idea (see also \cite[Lemma 5.5]{MR2304157}), we obtain
\begin{equation}
                            \label{eq9.40}
\|u\|_{\cH_{p,q,w,0}^{\alpha,1}(\bR^d_T)} \leq N \|g_i\|_{L_{p,q,w}(\bR^d_T)} + N \|f\|_{L_{p,q,w}(\bR^d_T)}
+N \|u\|_{L_{p,q,w}(\bR^d_T)},
\end{equation}
where $N = N(d,\delta,\alpha, p,q,K_1,K_0,R_0)$.
To get rid of the $u$ term on the right-hand side of \eqref{eq9.40} and conclude the estimate \eqref{eq1124_03}, we use the same time-partition argument as in the proof of  \cite[Theorem 2.4]{MR4030286}. Finally, the solvability of the equation follows from the a priori estimate \eqref{eq1124_03} and the method of continuity.
\end{proof}

\begin{proof}[Proof of Theorem \ref{thm1122_2}]
As before, we first prove
\begin{equation}
							\label{eq1124_07}
\|D^2u\|_{L_{p,q,w}(\bR^d_T)} \leq N \|f\|_{L_{p,q,w}(\bR^d_T)}
\end{equation}
for $u \in \bH_{p,q,w,0}^{\alpha,2}(\bR^d_T)$ having compact support in $Q_{R_0}(t_1,0)$ and satisfying \eqref{eq1124_01} with $b^i = c = 0$. Using Lemma \ref{lem1124_2} as in the proof of Theorem \ref{thm1122_1}, we have
\begin{equation}
							\label{eq1124_08}
\|DD_{x'}u\|_{L_{p,q,w}(\bR^d_T)} \leq N \|f\|_{L_{p,q,w}(\bR^d_T)} + N \gamma_0^{1/(\nu p_0)}\|D^2u\|_{L_{p,q,w}(\bR^d_T)}.
\end{equation}
To complete the proof, that is, to have an estimate for $D_1^2u$, we write the equation as
\[
-\partial_t^\alpha u + a^{11}D_{11}u + \Delta_{x'}u = \Delta_{x'}u - \sum_{(i,j) \neq (1,1)} a^{ij}D_{ij} u + f.
\]
Set $u_1 = D_1u$, which satisfies the divergence type equation
\begin{equation}
							\label{eq1124_05}
-\partial_t^\alpha u_1 + D_1 \left(a^{11}D_1 u_1\right) + \Delta_{x'} u_1 = D_1 g_1,
\end{equation}
where
\[
g_1 = \Delta_{x'}u - \sum_{(i,j) \neq (1,1)} a^{ij}D_{ij} u + f.
\]
By applying Theorem \ref{thm1122_1} to \eqref{eq1124_05}, we get
\[
\|D_1^2u\|_{L_{p,q,w}(\bR^d_T)}=\|D_1u_1\|_{L_{p,q,w}(\bR^d_T)} \leq N \|g_1\|_{L_{p,q,w}(\bR^d_T)}
\]
\[
\leq N\|DD_{x'}u\|_{L_{p,q,w}(\bR^d_T)} + N\|f\|_{L_{p,q,w}(\bR^d_T)}.
\]
This combined with \eqref{eq1124_08} proves \eqref{eq1124_07} with a sufficiently small $\gamma_0$.

Now we remove the small support condition on $u$ and the condition $b^i = c = 0$ by using a partition of unity argument as in \cite[Corollary 5.4]{MR4186022}. We then get
\begin{equation}
                            \label{eq9.50}
\|u\|_{\bH_{p,q,w,0}^{\alpha,2}(\bR^d_T)} \leq N \|f\|_{L_{p,q,w}(\bR^d_T)}
+N \|u\|_{L_{p,q,w}(\bR^d_T)},
\end{equation}
where $N = N(d,\delta,\alpha, p,q,K_1,K_0,R_0)$.
To get rid of the $u$ term on the right-hand side of \eqref{eq9.50} and conclude the estimate \eqref{eq1124_04}, we use the same time-partition argument as in the proof of \cite[Theorem 2.2]{MR4186022}. Finally, the solvability of the equation follows from the a priori estimate \eqref{eq1124_04} and the method of continuity.
\end{proof}

\appendix

\section{}
\label{appendix1}

Let $t_0 \in \bR$, $0 < \nu < \mu$, $u \in L_p\left((-\infty,t_0) \times \Omega\right)$, $1 < p < \infty$, and
\begin{equation}
							\label{eq0212_07}
\eta(t) =
\left\{
\begin{aligned}
1 \quad &\text{if} \quad t \geq t_0 - \nu,
\\
0 \quad &\text{if} \quad t \leq t_0 - \mu,
\end{aligned}
\right. \quad |\eta'(t)|\leq \frac{2}{\mu - \nu}.
\end{equation}
Set
\begin{equation}
							\label{eq0212_01}
G(t,x) = \frac{\alpha}{\Gamma(1-\alpha)} \int_{-\infty}^t (t-s)^{-\alpha-1}\left(\eta(s) - \eta(t)\right) u(s,x) \, ds.
\end{equation}
Also set $\{s_k\}$ to be a sequence such that
\begin{equation}
							\label{eq0212_04}
s_1 = \mu, \quad s_k + \mu \leq s_{k+1}, \quad \frac{s_{k+1}}{N_0} \leq s_k,
\end{equation}
where $N_0 > 0$.

\begin{lemma}
							\label{lem0212_1}
For $G$ defined in \eqref{eq0212_01} with $\{s_k\}$ above, we have
\begin{multline}
							\label{eq0214_05}
\|G\|_{L_p\left((t_0-\mu, t_0)\times \Omega\right)} \leq \frac{2 \alpha \mu^{1-\alpha}}{(1-\alpha)\Gamma(1-\alpha)(\mu-\nu)} \left(\int_{t_0 - 2 \mu}^{t_0-\nu} \int_\Omega |u(t,x)|^p \, dx \, dt\right)^{1/p}
\\
+ \frac{\alpha (N_0^{\alpha+1}+1)}{\Gamma(1-\alpha)} \sum_{k=1}^\infty s_k^{-\alpha-1} (s_{k+1}-s_k)^{1-1/p} \mu^{1/p} \left( \int_{t_0 - s_{k+1}}^{t_0-s_k} \int_\Omega |u(s,x)|^p \, dx \, ds \right)^{1/p}.
\end{multline}
\end{lemma}

\begin{remark}
							\label{rem0214_1}
It is worth noting that the right-hand side of the inequality \eqref{eq0214_05} does not contain the integral of $u$ with respect to $t$ on $(t_0-\nu,t_0)$ because $\eta(t) = 1$ for $t \geq t_0 - \nu$.
\end{remark}

\begin{proof}[Proof of Lemma \ref{lem0212_1}]
Note that
\[
\frac{\Gamma(1-\alpha)}{\alpha}G(t,x) = \int_{-\infty}^t (t-s)^{-\alpha-1}\left(\eta(s)-\eta(t)\right) u(s,x) 1_{s \leq t_0 - \nu} \, ds
\]
because $\eta(s)-\eta(t)=0$ for $s \in (t_0-\nu,t_0)$ and $t \ge s$.
We then write
\begin{multline}
							\label{eq0212_05}
\frac{\Gamma(1-\alpha)}{\alpha}G(t,x) = \left(\int_{t-\mu}^t + \int_{-\infty}^{t-\mu}\right) (t-s)^{-\alpha-1} \left(\eta(s) - \eta(t)\right) u(s,x) 1_{s \leq t_0 - \nu} \, ds
\\
=: I_1(t,x) + I_2(t,x).
\end{multline}
Since
\begin{align*}
&|I_1(t,x)| \leq \frac{2}{\mu-\nu} \int_{t-\mu}^t (t-s)^{-\alpha} |u(s,x)| 1_{s \leq t_0 - \nu} \,ds\\
&= \frac{2}{\mu-\nu} \int_0^{\mu} s^{-\alpha}|u(t-s,x)| 1_{t-s \leq t_0 - \nu} \,ds,
\end{align*}
by the Minkowski inequality, we have
\begin{equation}
							\label{eq0212_06}
\|I_1\|_{L_p\left((t_0-\mu,t_0)\times \Omega\right)} \leq \frac{2 \mu^{1-\alpha}}{(1-\alpha) (\mu-\nu)} \left(\int_{t_0 - 2 \mu}^{t_0 - \nu} \int_\Omega |u(t,x)|^p \, dx \, dt\right)^{1/p}.
\end{equation}
To estimate $I_2$, we note that $\eta(s)=0$ for $s \leq t - \mu = t - s_1$ and $t \in (t_0-\mu,t_0)$, which implies that, for $t \in (t_0-\mu,t_0)$,
\begin{align*}
&|I_2(t,x)| \leq \int_{-\infty}^{t-s_1} (t-s)^{-\alpha-1} |u(s,x)|\,ds\\
&= \sum_{k=1}^\infty \int_{t-s_{k+1}}^{t-s_k} (t-s)^{-\alpha-1} |u(s,x)|\,ds \leq \sum_{k=1}^\infty s_k^{-\alpha-1} \int_{t-s_{k+1}}^{t-s_k} |u(s,x)|\,ds\\
&\leq \sum_{k=1}^\infty s_k^{-\alpha-1} \int_{t_0-\mu-s_{k+1}}^{t_0-s_k} |u(s,x)|\,ds \leq \sum_{k=1}^\infty s_k^{-\alpha-1} \int_{t_0-s_{k+2}}^{t_0-s_k} |u(s,x)|\,ds\\
&\leq (N_0^{\alpha+1}+1) \sum_{k=1}^\infty s_k^{-\alpha-1} \int_{t_0-s_{k+1}}^{t_0-s_k} |u(s,x)|\,ds,
\end{align*}
where in the last inequality we used the condition $s_{k+1}\leq N_0 s_k$ in \eqref{eq0212_04}.
Thus, by H\"older's inequality and the Fubini theorem,
\begin{align*}
&\|I_2\|_{L_p\left((t_0-\mu,t_0)\times \Omega\right)}\\
&\leq (N_0^{\alpha+1}+1) \sum_{k=1}^\infty s_k^{-\alpha-1} (s_{k+1}-s_k)^{1-1/p} \mu^{1/p} \left( \int_{t_0 - s_{k+1}}^{t_0-s_k} \int_\Omega |u(s,x)|^p \, dx \, ds \right)^{1/p}.
\end{align*}
We obtain the inequality in the lemma from this inequality, \eqref{eq0212_05}, and \eqref{eq0212_06}.
\end{proof}

\bibliographystyle{plain}

\def\cprime{$'$}

\end{document}